\documentclass[11pt]{article}
\usepackage{amsfonts,comment}
\usepackage{amssymb}
\usepackage{graphicx}
\usepackage{mathrsfs}
\usepackage{amsmath}
\usepackage{amsthm}

\usepackage[affil-it]{authblk}
\usepackage{xcolor}

\nonstopmode

\newtheorem{theorem}{Theorem}[section]
\newtheorem{definition}[theorem]{Definition}
\newtheorem{proposition}[theorem]{Proposition}
\newtheorem{corollary}[theorem]{Corollary}
\newtheorem{lemma}[theorem]{Lemma}
\newtheorem{assumption}[theorem]{Assumption}

\newtheorem{notation}[theorem]{Notation}
\newtheorem{remark}[theorem]{Remark}
\newtheorem{example}[theorem]{Example}
\newtheorem{examples}[theorem]{Examples}

\newtheorem{foo}[theorem]{Remarks}
\newtheorem{Pre}[theorem]{}

\newcommand{\Ea}{{\mathbb E}}
\newcommand{\Prob}{{\mathbb P}}

\newcommand{\R}{{\mathbb R}}
\newcommand{\N}{{\mathbb N}}

\newcommand{\Acal}{{\mathcal A}}

\newcommand{\Ccal}{{\mathcal C}}
\newcommand{\Dcal}{{\mathcal D}}
\newcommand{\Ecal}{{\mathcal E}}
\newcommand{\Fcal}{{\mathcal F}}
\newcommand{\Gcal}{{\mathcal G}}
\newcommand{\Hcal}{{\mathcal H}}

\newcommand{\Lcal}{{\mathscr L}}

\newcommand{\Ncal}{{\mathcal N}}

\newcommand{\Scal}{{\mathcal S}}

\newcommand{\Ucal}{{\mathcal U}}

\newcommand{\Fact}{{\cl^{\infty}(\mathcal{F})}}
\newcommand{\FactI}{{\cl^{\infty}(\mathcal{F}_{t_1})}}

\renewcommand\inf{\qopname\relax m{\vphantom{p}inf}}

\newcommand{\cf}{\mathcal{F}}

\newcommand{\cn}{\mathcal{N}}

\usepackage{mathrsfs}
\newcommand{\cl}{\mathscr{L}}
\newcommand{\norm}[1]{\left\Vert#1\right\Vert}
\newcommand{\ob}{\mathbf{1}}

\begin{document}

\title{Stochastic Dynamic Utilities and  Intertemporal Preferences}

\author{Marco Maggis\thanks{Email: \texttt{marco.maggis@unimi.it}. The author would like to thank Marco Frittelli and Fabio Maccheroni for inspiring discussions on this subject.} \and Andrea Maran\thanks{Email: \texttt{andre.maran95@gmail.com}.} 
}

\date{\today}

\maketitle

\begin{abstract}
We propose an axiomatic approach which economically underpins the
representation of dynamic intertemporal decisions in terms of a stochastic dynamic
utility function, sensitive to the information available to the decision maker. Our construction is iterative and based on
time dependent preference connections, whose characterization is
inspired by the original intuition given by Debreu's State
Dependent Utilities (1960).

\end{abstract}




\section{Introduction}

The \emph{criterium} which leads the decisions of every agent,
intervenes in many aspects of real life, determining the
economical, political and financial dynamics. For this reason the psychological analysis and the mathematical axiomatization of the agents' behavior has gained a lot of interest, leading to a flourish
stream of research literature (see \cite{Ma15} for an exhaustive review). The first key element which comes into play in the decision process is the Subjective Probability, which has been intensively studied since the preliminary contributions by de Finetti \cite{DF31}. Von Neumann and Morgenstern \cite{vNM47} initiated the work on preferences over lotteries, which adimit a representation in terms of an expected utility. This intuition dates back to a paper published in 1738 (see \cite{Be54}), where Bernoulli 
already realized that any decision is heavily linked to the \textquotedblleft particular circumstances of the person making the estimate\textquotedblright, 
which could vary significantly depending on the observed evolution of information. 
For example a fund manager may start behaving in a risk seeking manner under the stress provoked by a plunge of the financial markets, which is causing severe losses.
\\Debreu \cite{De60} gave an axiomatic setup (on a finite state space) to model preference relations which can depend on the future state of nature and can be represented by the so-called state-dependent utility functions (see Theorem \ref{Debreu60}, in the
Appendix). State dependent preferences are sensitive to the random outcomes that may occur in the future and therefore the agent subjective utility may be affected by different future scenarios related to the occurrence of specific events. Karni \cite{Ka83} developed measures of risk aversion which allows the partial ordering of state dependend utilities in view of optimal risk sharing analysis. In \cite{WZ99}, Wakker and Zank provided an extension of Debreu's
result from finite to infinite dimension, for the special case of real-valued outcomes and monotonic preferences. The development of their extended functional, additively decomposable on
infinite-dimensional spaces, leads to a numerical representation
of the preferences in terms of a state dependent utility $u$
and a probability $\Prob$ (see Theorem \ref{formaintegrale},
Appendix). The main results in \cite{WZ99} (and \cite{CL06}) will indeed play a key role in the proofs of the results contained in the present paper.
\\ In \cite{KP78} Kreps and Porteus gave rise to a new axiomatic treatment of the temporal resolution of uncertainty. They consider a discrete time model $t=0,\ldots,T$  where an individual must choose an action $d_t$ constrained to the state $x_t$ occurred at time $t$.  
As a random event takes place determining an immediate payoff $z_t$, the action $d_t$ will affect the probability distribution of $(z_t,x_{t+1})$ where $x_{t+1}$ is the new state of the world. The result is a dynamic choice behaviour which cannot be represented by a single cardinal utility.
\\ Epstein Zin \cite{EZ89} and Duffie Epstein \cite{DE92} (see also \cite{EW94}) constructed a class of recursive preferences over intertemporal consumption lotteries respectively in discrete and continuous time models. In \cite{EZ89} the recursive utility at time $t$ is given by an aggregating function i.e. $V_t(c)=W(c_t, m_{t}(V_{t+1}))$ where $c_t$ is the consuption and $m_{t}(V_{t+1})$ the certainty equivalent at time $t$ of $V_{t+1}$. Similarly Duffie and Epstein \cite{DE92} obtained a representation of the recursive utility on consumption streams of the form 
\begin{equation*}
	V_t(c) = \Ea_{\Prob}\left[ \int_t^T\left( f(c_s,V_s(c)) + \frac{1}{2}A(V_s(c)) |\sigma_s|^2 \right) ds \big| \Fcal_t \right],
\end{equation*}  
where $f$ is an aggregator, $A$ is the variance multiplier and $\sigma$ is a volatility process.
In such context the system of \emph{conditional} preferences between two consumptions 
is determined by the recursive utility as follow:
\begin{equation*}
	c \succeq_{\omega,t} c' \iff V_t(c,\omega) \geq V_t(c',\omega).
\end{equation*}
In \cite{EL93}, Epstein and LeBreton  showed that the existence of a Bayesian prior is implied by preferences based on beliefs which admit a dynamically consistent updating in response to new information. The effect of consequences by the mean of conditional preferences over acts is introduced by Skiadas in \cite{Sk97}. Given an event $F$, the preference relation $x\succeq^F y$  \textquotedblleft has the
 interpretation that, ex ante, the decision maker regards the consequences of act
 $x$ on event $F$ no less desirable than the consequences of act $y$ on the same
 event \textquotedblright(\cite{Sk97} pp. 350). Wang \cite{Wa03} axiomatized three updating rules for a class of conditional preferences over consumption-information profiles.
 A systematic study of conditional preferences is provided in \cite{DJ14}: a \emph{conditional preference order} is a binary relation $\succeq$ which is reflexive, transitive and locally complete. An opportune extension to the conditional setup of the indipendence and Archimedean axioms, led in \cite[Theorem 5.2]{DJ14} to the representation of conditional preferences over the set of lotteries in terms of a conditional utility function. 
\\ Recurvive multiple-priors and dynamic variational preferences (see resp. \cite{ES03} and \cite{MMR06}) deals with conditional preference relations $\succeq_{t,\omega}$ on consumption streams $h$. Here $t\in \{0,1,2,\ldots,T\} $ is a point in time and $\omega$ is the path of the state space observed up to time $t$. Recursive multiple-priors utility and dynamic variation preferences can be represented respectively in the form of conditional functionals  
\begin{eqnarray}\label{consumption1}
V_t(h) & = & \inf_{P\in\Delta} \big(E_{P}\big[\sum_{\tau\geq t}\beta^{\tau-t}u(h_{\tau})\mid \Fcal_t\big]\big)
\\ V_t(h) & = & \inf_{P\in\Delta} \big(E_{P}\big[\sum_{\tau\geq t}\beta^{\tau-t}u(h_{\tau})\mid \Fcal_t\big]+c_{t}(p|\Fcal_t)\big),\label{consumption2} 
\end{eqnarray}
where $c_t$ is the recursive ambigity index, which under some restrictions gurantees time consistency of the preferences (See \cite[p.14, Axiom 4]{MMR06}). In both papers \cite{ES03,MMR06} the Dynamic Consistency Axiom plays a fundamental role and inspired the result contained in Proposition \ref{timecon}, Section \ref{ITP:main} of the present paper.       
\\ Finally we observe that in the recent paper \cite{RTV18}, Riedel et al. consider dynamic preferences $\succeq_{t,s}$ on couples $(P,f)$ where $P$ belongs to a set of probabilities and $f$ is an act (imprecise probabilistic framework). An important feature is that Dynamic Consistency of the preferences guarantees that the set of conditional priors is stable under pasting.

\paragraph{From Economics to Finance: the dynamics of decision making.} 
The interplay between Decion Theory and Financial Mathematics had its outbreak after the important contribution given by Merton in \cite{Me71} and is witnessed by the flourish literature on stochastic optimal control (see \cite{Ph09} for a detailed exposition). 
\\ The classical utility maximization problem can be formulated as a stochastic control problem of the form 
$$v(t,X)=\sup_{\alpha\in \Acal(t,X)} E_{\Prob}[u(V_T(t,X,\alpha))\mid \Fcal_t],$$
where the $\sup$ is to be intended as a $\Prob$ essential supremum, $\Acal(t,X)$ is the set of admissible strategies (starting at time $t$), $u$ is a concave utility function and $V_T(t,X,\alpha)$ is the $\Fcal_T$-measurable final payoff of the strategy $\alpha$ with initial random endowment $X$ (which is $\Fcal_t$-measurable). We may question when an agent, acting as a utility optimizer, is  willing to invest in a strategy $\alpha$ from time $t$ to time $T$, provided she owns at $t$ the random amount $X$. The answer to this question is deeply related to the intertemporal comparison between $X$ and the final value of the strategy $\alpha$ given by $V_T(t,X,\alpha)$. One rational solution could be that the agent enters the dynamic investment only if she believes to hold an optimal solution. Namely we can define an intertemporal relation $\succeq_{t,T}$  by 
\begin{eqnarray}\label{Merton}
X \succeq_{t,T}  V_T(t,X,\alpha) & \text{ if and only if } & v(t,X)\geq E_{\Prob}[u(V_T(t,X,\alpha))\mid \Fcal_t] \quad \Prob\text{-a.s.}.
\end{eqnarray}
The Dynamic Programming Principle (\cite[Theorem 3.3.1]{Ph09}) implies that for any bounded random variable $X$, $v(t,X)\geq E_{\Prob}[u(V_T(t,X,\alpha))|\Fcal_t]$ and equality holds 
whenever $\alpha^*$ is the optimal policy.
In this case $X$ is the intertemporal equivalent of $V_T(t,X,\alpha^*)$ which will be named in the following section of this paper Conditional Certainty Equivalent.  In this example, $v$ represents the indirect utility and the preference relation $\succeq_{t,T}$ is not a standard binary relation and its properties need to be introduced carefully as we shall do in an abstract fashion in Section \ref{ITP:main}.

This classical backward approach to
utility maximization has recently been argued in a series of paper
by Musiela and Zariphopoulou starting from \cite{MZ06,MZ09} and a novel
forward theory has been proposed: the utility function is
stochastic, time dependent and moves forwardly. In this theory, the
forward performance (which replaces the indirect utility of the
classic case) is built through the underlying financial market and
must satisfy some appropriate martingale conditions. Inspired by
this idea, Frittelli and Maggis \cite{FM11} studied the
conditional (dynamic) version of certainty
equivalent (as defined in \cite{Pratt}). The preliminary object is
a stochastic dynamic utility $u(t,x,\omega)$ - i.e. a stochastic
field - representing the evolution of the preferences of the
agent. The novelty in \cite{FM11} is that the (backward)
conditional certainty equivalent, represents the time-$s$-value of
the time-$t$-claim $X$, for $0\leq s\leq t<\infty $, capturing in
this way the intertemporal nature of preferences. Unfortunately
any axiomatization of intertemporal preferences, which could justify
the representation in terms of stochastic dynamic utilities, is
still missing in the literature and our aim is to fill
this gap.

\paragraph{The aim of this paper.} 
Indeed people are highly impatient when comparing present and future
outcomes and both emotion-based and cognitive-based mechanisms
contribute to intertemporal distortions. In \cite{ZU16},
Zauberman and Urminsky provide an overview of the psychological
determinants of intertemporal choice such as impulsivity, goal
completion and reward timing, different evaluation of the future
in terms of concreteness, time perception and many other features:

\begin{center}
\textquoteleft\textquoteleft \emph{In sum, these findings
establish that the way people perceive future time itself is an
important factor in how they form their intertemporal preferences
[...]
\\ What is common across the various factors influencing
intertemporal preferences is that all these mechanisms influence
the relative attractiveness of achieving a present goal compared
to a later more distant one.} \textquoteright \textquoteright (see
\cite{ZU16}, p. 139)
\end{center}

In this paper we aim at characterizing a family of intertemporal
preference relations which compare random payoffs whose realizations will be known at different points in time. 
We will introduce a set of conditional axioms which will lead to the representation of preference in terms of a Stochastic Dynamic Utility $u(t,x,\omega)$ and a Subjective Probability $\Prob$ on a general state space $\Omega$ (see Theorem \ref{main:theorem}), which can be rephrased as:
conditional to the available information, $g$ is prefered at time $s$ to $f$ at time $t$ if and only if 
$$u(s,g)\geq E_{\Prob}[u(t,f)\mid\Fcal_s] \quad \Prob-a.s.$$
The Stochastic Dynamic Utility turns out to be a random field adapted to a given filtration which represents the information flow. For this reason $u(t,x,\omega)$ randomly reacts whenever the
Decision Maker becomes aware of new sensitive data, such as market behavior, news, catastrophic shocks or any other macro/micro factor which leads to a reconsideration of personal beliefs. 
Since different random payoffs are defined on different instants in time, 
the notion of preference relation\footnote{We point out that the use of the term \textquoteleft preference\textquoteright is slightly improper as the ordering will not be a binary relation as it is usually intended.} we are going to introduce will satisfy non-standard axioms and will take the name of Intertemporal Preferences (ITP). 

\medskip

The main novelty of our approach is that we provide an abstract axiomatization of Intertemporal Preferences which allows to include in our model \textquotedblleft the relative attractiveness of achieving a present goal compared to a later more distant one \textquotedblright. Indeed our iterative construction leads to an automatic forward updating of preferences depending on the avaliable information, which satisfies a form of dynamic consinstency.
As a byproduct we obtain a theoretical framework where the theory of Forward Performances \cite{AZZ18,MZ06,MZ09} and the study of Conditional Certainty Equivalent \cite{FM11} can be embedded.    

\medskip

The key ingredients of ITP can be
summarized by four elements: first the information at each time
is described by the existence of a filtration $\{\mathcal{F}_t
\}_{t \in [0,+\infty)}$, i.e. a family of sigma algebras such that
$\Fcal_s\subseteq \Fcal_t$ for $s\leq t$. Second, as ITP compares
random payoffs which live at different times, we shall need to
introduce a relation $\preceq_{s,t}$ (resp. $\succeq_{s,t}
$) for $s<t$ being two points in time. In particular
$g\preceq_{s,t} f$ will mean that the $\Fcal_t$-measurable payoff
$f$ (which will be fully revealed at time $t$) is preferred to the
$\Fcal_s$-measurable $g$, conditioned to the knowledge of the
information available at time $s$ (Similar for $g\succeq_{s,t}
f$). Third the preference relation $\preceq_{s,t}$ is not total if
the full information $\Fcal_s$ is not yet disclosed. The notion
of conditional preferences as introduced in \cite{DJ14} becomes
therefore an important tool to understand the nature of ITP.
Finally we will assume that the agent observes real information 
only through a discretisation of the time line, namely $t_0 = 0
<t_1 < ... < t_n <\ldots$. We observe that in \cite{DJ14} a
probability on the conditional sigma algebra was assumed to exist
a priori. In our approach this requirement is not necessary, but
we rather derive step by step a new probability update which follows directly from the decision theory structure we are choosing.

\medskip

The paper is structured as follows: in Section \ref{ITP:examples} we provide a description of the notations used in the paper and a toy example to motivate our study; Section \ref{ITP:main} is devoted to the description
of the set of axioms characterizing ITP and to the statement of the main
representation result. In Section \ref{unconditioned:updating} we
prove the result in the unconditioned case (i.e. for trivial
initial information). The aim of Section
\ref{unconditioned:updating} is twofold: on the one hand it will serve
as initial step of the induction argument we present in Section
\ref{inductive:proof} to obtain the complete proof of our main
Theorem \ref{main:theorem}. On the other hand, it is written in a
self-contained manner, so that it can be read and understood
independently from the general conditional setting.


\section{Preliminaries on Intertemporal preferences.}
\label{ITP:examples}

\subsection{Notations}\label{notation}
Throughout the paper we shall make use of the notations described in this short section. 
We fix a measure space $(\Omega,\Fcal)$ where $\Omega$
is the set of all possible events (\emph{state space}) and $\Fcal$ is a sigma
algebra. We shall model information
over time by the existence of an arbitrary filtration
$\{\mathcal{F}_t \}_{t \in [0,+\infty)}$, with $\Fcal_s\subseteq \Fcal_t\subseteq
\Fcal$ for every $s\leq t$. For any given sigma algebra $\Gcal\subseteq
\Fcal$ we denote by $\cl(\Gcal)$ the space of
$\Gcal$-measurable functions taking values in $\R$ (\emph{outcome space}). We shall usually refer to elements $f\in \cl(\Gcal)$ as random variables (or acts) and denote by $\cl^{\infty}(\Gcal)$ its
subspace collecting bounded elements i.e. $f\in \cl(\Gcal)$ such that $|f(\omega)|\leq k$ for any $\omega\in\Omega$ and some $k\geq 0$. On $\cl(\Gcal)$ and
$\cl^{\infty}(\Gcal)$ we shall consider the usual
pointwise order $f\leq g$ if and only if $f(\omega)\leq g(\omega)$
for every $\omega\in \Omega$ and similarly $f<g$ if and only if
$f(\omega)< g(\omega)$ for every $\omega\in \Omega$.  Given two
elements $f,g\in \cl^{\infty}(\Gcal)$ we use the notation
$f\vee g$, $f\wedge g$ to indicate respectively the minimum and
the maximum between $f$ and $g$. For a countable family of acts
$\{f_n\}_{n\in\N}\subseteq \cl^{\infty}(\Gcal)$ we
consider the $\inf_n f_n, \sup_n f_n $ the pointwise
infimum/supremum of the family and recall that if the family is
uniformly bounded then $\inf_n f_n, \sup_n f_n $ are elements of
$\cl^{\infty}(\Gcal)$. $\cl^{\infty}(\Gcal)$
endowed with the sup norm $\|\cdot\|_{\infty}$ becomes a Banach
lattice, where $\|f \|_{\infty}= \sup_{\omega\in
\Omega}|f(\omega)|$. By $\mathbf{1}_A$, $A\in\Gcal$ we indicate
the element of $\cl^{\infty}(\Gcal)$ such that
$\mathbf{1}_A(\omega)=1$ if $\omega\in A$ and $0$ otherwise. For
$f\in \cl^{\infty}(\Gcal)$ and $A\in \Gcal$,
$f\mathbf{1}_{A}$ denotes the restriction of $f$ to $A$; for any
couple $f,g\in \cl(\Gcal)$ and event $A\in \Gcal$, $f\mathbf{1}_{A} + g\mathbf{1}_{A^c}$
denotes the random variable that agrees with $f$ on $A$ and with $g$ on $A^c$.
Let $\Gcal_1\subset \Gcal_2$ be two sigma algebras. For a finite partition $\{A_1, ... , A_n \}\subset \Gcal_2$ of
$\Omega$ and $\{g_j\}_{j=1}^n\subset \Lcal^{\infty}(\Gcal_1)$,
$\sum_{j=1}^n g_j \mathbf{1}_{A_j}$ denotes the element assigning
$g_j$ on $A_j$, $\forall j = 1, ... , n$. This type of random variables can be interpreted as 
simple act conditional to $\Gcal_1$ and 
$\mathcal{S}_{\Gcal_1}(\Gcal_2)$ denotes the space conditional simple acts. The standard notion of simple acts can be
obtained when $\Gcal_1=\{\emptyset, \Omega\}$ and the
corresponding space will be denoted by $\mathcal{S}(\Gcal_2)$.

Whenever a probability $\Prob$ is given $(\Omega,\Fcal,\Prob)$
becomes a measure space and as usual we shall say that a
probability $\widetilde{\Prob}$ is dominated by $\Prob$
($\widetilde{\Prob}\ll \Prob$) if $\Prob(A)=0$ implies
$\widetilde{\Prob}(A)=0$ for $A\in \Fcal$. Similarly a probability
$\widetilde{\Prob}$ is equivalent to $\Prob$
($\widetilde{\Prob}\sim \Prob$) if $\Prob\ll\widetilde{\Prob}$ and
$\widetilde{\Prob}\ll \Prob$. A property holds $\Prob$ almost
surely ($\Prob$-a.s.), if the set  where it fails has $0$
probability.
\\For any given sigma algebra
$\Gcal\subseteq \Fcal$ we shall denote with $L^{0}(\Omega
,\mathcal{G},\mathbb{P})$ the space of equivalence classes of
$\mathcal{G}$ measurable random variables that are $\mathbb{P}$
almost surely equal and by $L^{\infty}(\Omega
,\mathcal{G},\mathbb{P})$ the subspace of ($\Prob$ a.s.) bounded
random variables. Formally any $f\in \cl(\Gcal)$ will be a representative of the class $X:=[f]_{\Prob}\in L^{0}(\Omega,\Gcal,\Prob)$. Moreover the essential ($\mathbb{P}$ a.s.)
\emph{supremum} of an arbitrary family of random variables
$\{X_{\lambda}\}_{\lambda\in\Lambda}\subseteq L^{0}(\Omega
,\mathcal{G},\mathbb{P})$ will be simply denoted by $\Prob-\sup
\{X_{\lambda }\mid \lambda\in\Lambda\}$, and similarly for the
essential \emph{infimum} (see \cite{FS04} Section A.5 for
reference). 
\\Let us fix $(\Omega,\Gcal,\Prob)$: given  a random field $\phi:\Omega\times \R\to \R$ such that for every $f\in \cl^{\infty}(\Gcal)$ the map $\omega\mapsto \phi(\omega,f(\omega))$ is $\Gcal$-measurable (see \cite{Rock} for further details) we introduce the notation
\begin{equation}\label{range}
L(\Gcal;\phi) = \{[\phi(\cdot,f(\cdot))]_{\Prob}\mid f\in \cl^{\infty}(\Gcal)\}.   
\end{equation}
Indeed $L(\Gcal;\phi)$ represents the range of the random field $\phi$ in the space $L^{0}(\Omega,\mathcal{G},\mathbb{P})$.
In order to tackle the issue of continuty of the conditional representation of preferences,
we need to introduce an \emph{ad hoc} definition of continuity for stochastic fields. 
Consider $(\Omega,\Gcal,\Prob)$ and $\phi:\Omega\times \R\to \R$ as for \eqref{range}, we say that $\phi$ is
$\star$-continuous if $\forall f \in \Lcal^\infty(\Gcal)$ it holds that $f(\omega)$ 
belongs to the points of continuity of $\phi(\cdot, \omega)$ for $\Prob$-a.e. $\omega \in \Omega$ (see Definition \ref{*-continuity} in Appendix \ref{Appendix A} for the formal statement).
\\ Finally the space of $\Prob$ integrable random
variables will be denoted by $L^1(\Omega
,\mathcal{G},\mathbb{P})$. We use the standard notation and
indicate by $E_{\Prob}[X]$ the Lebesgue integral of $X\in
L^1(\Omega ,\mathcal{G},\mathbb{P})$. Moreover if $\Hcal$ is a
sigma algebra contained in $\Gcal$ then $E_{\Prob}[X\mid \Hcal]$
denotes the conditional expectation of $X$ given $\Hcal$ and
$\Prob_{|\Hcal}$ the restriction of the probability $\Prob$ on the
smaller sigma algebra $\Hcal$.

\subsection{State dependent utility and the role of information: a toy example} \label{example}
Two brothers $E,Y$ are inheriting from
their old and rich grandmother. The elder brother $E$ is asked to
choose between receiving $1$ million Euros immediately (at time
$t=0$), or waiting two years (time $t_2$) when his grandmother
will move to the rest home in Sardinia and earn her wonderful
villa near the Como Lake. Alternatively $E$ could wait until the
intermediate time $t_1$ to make up his decision, but in any case
the younger brother $Y$ will have to accept what is left from $E$
after his decision is taken.
\\ The value of the villa at time $0$ is equal to
$1$ million, but of course it makes little sense to compare the
two values today since the villa will be available only at $t_2$.
\\ Now assume that at time $t_1$ election for the new Italian Government
will take place and the catastrophic event of Italy leaving the
European Union (with a consequent default of its economic system)
may occur. Call this event $A$ and set $\Fcal_{t_1}=\{\emptyset,
\Omega, A,A^c\}$. Brother $E$ knows that if $A^c$ will occurs the
value of the villa will increase to $1.11\cdot 10^6$, but in case
of default it will fall down to $2\cdot 10^5$. The probability of
the default event $A$ is low but not negligible, say
$\Prob(A)=0.01$. Finally the probability of defaulting at time
$t_2$ (call this event $D$) knowing that $A^c$ occurred is almost
negligible, for instance $\Prob(D\mid A^c)=10^{-6}$ (in which case
the villa would be worth again $2\cdot 10^5$). In case that a
default did not occur neither at time $t_1$ nor at time $t_2$ then
the value of the villa at $t_2$ would jump up to $1.8\cdot 10^6$.
Information at time $t_2$ is therefore described by $\Fcal_{t_2}$
the sigma algebra generated by $\{A,D\}$.

Agent $E$ is assumed to be risk neutral as far as Italy is not
defaulting i.e. $u(x)=x$. In case of a default (either at time
$t_1$ or $t_2$) his utility function would be
$\tilde{u}(x)=\frac{1}{2}x$ if $x\geq 0$ or $\tilde{u}(x)= 2x$ if
$x< 0$ . The naive idea is that once the default has occurred the
agent gives more importance in avoiding losses, rather than gaining
money. We can synthesize this reasoning by introducing the
stochastic dynamic utility as follows

$$u(t,x,\omega) = \left\{ \begin{array}{cc} u(x) & \text{ if } t=0
\\ \tilde{u}(x)\mathbf{1}_{A}(\omega)+u(x) \mathbf{1}_{A^c} (\omega) & \text{ if } t=1
\\ \tilde{u}(x)\mathbf{1}_{A\cup D} (\omega) +u(x) \mathbf{1}_{A^c\cap D^c} (\omega) & \text{ if } t=2
\end{array}\right.$$

We make the following considerations.

\begin{itemize}

\item If agent $E$ compares the choice between getting $10^6$
today or the villa at time $t_2$, then he is comparing the utility
$u_0(10^6)=10^6$ with respect to the expected utility of the
payoff at time $t_2$ given by
$$\text{Expected payoff } = 1.8\cdot 10^6 \cdot (1-10^{-2}-10^{-6})+ \frac{1}{2} \cdot 2 \cdot 10^5 \cdot (10^{-2}+10^{-6}).$$
This Expected payoff is strictly greater than $10^6$ and indeed if
$E$ neglects the intermediate time $t_1$ then he will choose for
the villa instead of immediate money. But this impulsive strategy
would not lead to an optimal solution.

\item Assume now that the agent first compare $10^6$  with the
value of the villa at time $t_1$. Then
$$\text{Expected payoff } = 1.11 \cdot 10^6 \cdot 0.9 + \frac{1}{2} \cdot 2 \cdot 10^5 \cdot 0.01= 10^6. $$
This means that the expected value of the villa at time $t_1$ is
the same of the cash amount of money which means that $E$ is
indifferent between taking the decision today ($t=0$) or tomorrow
($t_1$). Therefore he has better waiting until the elections take
place and distinguish between event $A$ or $A^c$. In the former
case $E$ will choose $10^6$ which is in fact better than the value
of the villa. In the second case he will prefer obtaining the
villa at time $t_2$ rather than $10^6$  at time $t_1$. Clearly
this second strategy provides an optimal final profile, since it
exploits the additional intermediate information.
\end{itemize}
\begin{remark}\label{null:info}
Notice that the reasoning would change if the elder brother reckons $\Prob(A)=0$. In such a case the additional intermediate information would play no role in the decision process.  
\end{remark}

\section{An axiomatization of intertemporal preferences.}\label{ITP:main} 

We consider a
time interval $[0,+\infty)$, together with a fixed (countable)
family of updating times $t_0 = 0 < t_1 < \ldots < t_n <\ldots,$.
At each $t_i$ the agent shall reconsider her preference relations
depending on the observed information. In particular at time $t_0
= 0$ no information is available, i.e. $\mathcal{F}_0 = \{
\emptyset, \Omega\}$. Information at each time $t$ is represented
by a sigma algebra $\Fcal_{t}$ and since information increases in time we shall have $\Fcal_{s}\subseteq \Fcal_{t}$ for every $s\leq t$.
\\In the entire paper acts are intended as real valued random variables, matching the framework used in \cite{WZ99}\footnote{Indeed this choice 
is not \textquoteleft without loss of generality\textquoteright. Nevertheless as explained in the Introduction this research is inspired by 
potential financial applications and therefore we prefer to choose a more financial friendly setup.}.

\begin{assumption}\label{overall:assumption} We shall always assume throughout the paper that the agent is endowed by some initial utility function $u_0 : \mathbb{R} \to
\mathbb{R}$ which is strictly increasing and continuous (not necessarily
concave). For simplicity we will consider the case $u_0(0)=0$.
\end{assumption}

The paper could be developed without fixing $u_0$ as in Assumption \ref{overall:assumption}. The advantage of this choice is twofold: on the one hand fixing a single $u_0$ gives a sharper uniqueness result in Theorem \ref{main:theorem} (see also Remark \ref{remark:u0}). On the other hand $u_0$ plays the role of \textquotedblleft initial value \textquotedblright, which follows from the idea that $u_0$ is inherited from the attitude towards decisions shown by the agent in the past. The shape of $u_0$ is effective: for instance it allows to understand if at the initial time the agent is risk averse or risk seeking and in general how she evaluates variation in the amount of money she owns (quoting  \cite{Be54} \textquotedblleft Thus there is no doubt that a gain of  one thousand ducats is more significant to a pauper  than to a rich man though both gain the same  amount.\textquotedblright). 

\medskip

The time $t_1$ represents the first instant when the Decion Maker
observes available information which will potentially influence
her decision. Random payoffs at time $t_1$ are described by random variables in $\cl^{\infty}(\mathcal{F}_{t_1})$ and the
agent compares these random payoffs with initial sure positions
represented by elements in $\R$. In Section
\ref{unconditioned:updating} we shall provide the representation
of an intertemporal preference $\succeq_{0,1}$  connecting the initial time
$t_0=0$ to $t_1$. In Proposition \ref{main:0} we will show the following: if $\succeq_{0,1}$ is complete, transitive, monotone, continuous and satisfies the Sure-Thing Priciple then 
for any $f\in \cl^{\infty}(\mathcal{F}_{t_1})$ and $a\in\R$
we have $a\succeq_{0,1} f$ is and only if $u_0(a)\geq \int_{\Omega}u_1(f(\omega),\omega)d\Prob_1(\omega)$. 
This representation is based on Theorem \ref{formaintegrale} by Wakker and Zank and shows how new inputs will affect the attitude of an agent towards decisions, generating a new utility $u_1$ which will
depend on the state of nature realized. Once time $t_1$ is reached
the Decison Maker will start considering a new aim in the next future, say
$t_2$, and compare random payoffs, known at time $t_1$, with those
which will depend on events occurring at $t_2$. From the
$t_0$ perspective the new intertemporal preference $\preceq_{1,2}$
will be a conditional preference relation which incorporates the
further knowledge reached at time $t_1$. Therefore we shall follow the idea
proposed by \cite{DJ14} and make use of similar techniques
developed in the conditional setting. This procedure will repeat iteratively at every interval from $t_i$ to $t_{i+1}$ and for this reason our main
result will be proved by induction over updating times. Each
updating step from $t_i$ to $t_{i+1}$ will be characterized by a
preference interconnection $\preceq_{i,i+1}$ (or
$\succeq_{i,i+1}$) satisfying conditional transition axioms. Of
course since the proof proceeds by induction we will assume that
we reached the desired representation up to step $t_i$ and show
the representation at the succeeding time $t_{i+1}$. This will
guarantee the existence of a probability $\Prob_i$ only on the
sigma algebra $\Fcal_{t_i}$, which we will need to update to the
larger sigma algebra $\Fcal_{t_{i+1}}$, following the Bayesian
paradigm.

\bigskip

For the statement of Theorem \ref{main:theorem}, we fix an
arbitrary $N$ and a family of intertemporal preference relations
$\succeq_{i,i+1}$ for $i=0,\ldots,N-1$ with the following meaning:
for any $g \in \cl^{\infty}(\mathcal{F}_{t_i})$ and $f \in
\cl^{\infty}(\mathcal{F}_{t_{i+1}})$ we say that $g
\succeq_{i,i+1} f$ if the agent prefers to hold the gamble $g$ at
time $t_i$ than the gamble $f$ at time $t_{i+1}$, knowing all the
information provided at time $t_i$ (similarly for $g
\preceq_{i,i+1} f$).
\\ As usual we say that $g \in \cl^{\infty}(\mathcal{F}_{t_i})$
is equivalent to $f \in
\cl^{\infty}(\mathcal{F}_{t_{i+1}})$, namely $g
\sim_{i,i+1} f$, if both $g \succeq_{i,i+1} f$ and $g
\preceq_{i,i+1} f$, and define the family of \emph{null events} for
every $i=1,\ldots, N$ as
\begin{equation}\label{null:events}\mathcal{N}(\mathcal{F}_{t_{i}}) = \{ A \in
\mathcal{F}_{t_{i}} : g \sim_{i-1,i} f \Rightarrow g \sim_{i-1,i}
\tilde g \mathbf{1}_A + f \mathbf{1}_{A^c}, \forall f, \in
\cl^{\infty}(\mathcal{F}_{t_{i}}), g,\tilde g \in
\cl^{\infty}(\mathcal{F}_{t_{i-1}}) \}.
\end{equation}
An event $A\in \Fcal_{t_i}$ is called \emph{essential} at time
$t_i$ if $A\in \mathcal{F}_{t_i}\setminus
\mathcal{N}(\mathcal{F}_{t_{i}})$

\paragraph{The Transition Axiom.} We are now ready to introduce the first axiom characterizing the Intertemporal Preferences. In this context the preference ordering $\succeq_{i,i+1}$ is not anymore 
a binary relation as it is generally understood. For this reason we shall need a reformulation of the axioms which shall be compared to more classical ones. Moreover we work in a conditional setting, which means that 
the relation $\succeq_{i,i+1}$ is assessed taking into account information available at time $t_i$. Information are modelled by measurable sets $A\in\Fcal_{t_i}$ and in addition the Decision Maker has a subjective belief concerning sets which are relevant ($A\in \mathcal{F}_{t_i}\setminus \mathcal{N}(\mathcal{F}_{t_{i}})$) and those which are irrelevant ($A\in \mathcal{N}(\mathcal{F}_{t_{i}})$). To understand the central role of null sets we refer to the example contained in Section \ref{example} (see in particular Remark \ref{null:info}).

\begin{description}\item[(T.i)] Transition Axiom for the couple $\preceq_{i,i+1}$, $\succeq_{i,i+1}$. Let $A,B\in \Fcal_{t_i}$, $g \in
\cl^{\infty}(\mathcal{F}_{t_i})$ and $f \in
\cl^{\infty}(\mathcal{F}_{t_{i+1}})$ then we require for
$\preceq_{i,i+1}$, $\succeq_{i,i+1}$ to be

\begin{enumerate}

\item[1.] locally complete: there exists $A \in
\mathcal{F}_{t_i}\setminus \mathcal{N}(\mathcal{F}_{t_{i}})$ such
that either $g \mathbf{1}_A \succeq_{i,i+1} f \mathbf{1}_A$ or
$g\mathbf{1}_A \preceq_{i,i+1} f \mathbf{1}_A$.

\item[2.] transitive: if $g \succeq_{i,i+1} f$ and $h
\preceq_{i,i+1}f$ then $\{g<h\}\in
\mathcal{N}(\mathcal{F}_{t_{i}})$;

\item[3.] normalized: if $A,B\in \mathcal{N}(\mathcal{F}_{t_{i}})$
then $\mathbf{1}_{A}\sim_{i,i+1} \mathbf{1}_{B}$.

\item[4.] non-degenerate: for any $f \in
\cl^{\infty}(\mathcal{F}_{t_{i+1}}) $ there exist
$g_1,g_2\in \cl^{\infty}(\mathcal{F}_{t_{i}}) $ such that
$g_1\preceq_{i,i+1} f$ and $g_2\succeq_{i,i+1} f$.

\item[5.] consistent: if $g \mathbf{1}_A \succeq_{i,i+1} f
\mathbf{1}_A$ (resp. $\preceq_{i,i+1}$) and $B \subseteq A$ then
$g \mathbf{1}_B \succeq_{i,i+1} f \mathbf{1}_B$ (resp.
$\preceq_{i,i+1}$);

\item[6.] stable: if $g \mathbf{1}_A \succeq_{i,i+1} f
\mathbf{1}_A$ (resp. $\preceq_{i,i+1}$) and $g \mathbf{1}_B
\succeq_{i,i+1} f \mathbf{1}_B$ (resp. $\preceq_{i,i+1}$) then $g
\mathbf{1}_{A \cup B} \succeq_{i,i+1} f \mathbf{1}_{A \cup B}$
(resp. $\preceq_{i,i+1}$);

\end{enumerate}
\end{description}

\noindent Before giving an explanation of (T.i) in its full generality,
we specialize it to the unconditioned case ($i=0$). 

\begin{description}\item[(T.0)] Transition preference
relation $\preceq_{0,1}$.
\begin{enumerate}

\item[1.] complete: for $a \in \mathbb{R}$ and $f \in
\cl^{\infty}(\mathcal{F}_{t_1})$ either $a \succeq_{0,1}
f$ or $a \preceq_{0,1} f$;

\item[2.] transitive: $a \preceq_{0,1} f$ and $b \succeq_{0,1} f$
implies $a\leq b$;

\item[3.] normalized: $0\sim_{0,1} 0$ (i.e. $0\succeq_{0,1} 0$ and
$0\preceq_{0,1} 0$).

\item[4.] non-degenerate: for any $f \in
\cl^{\infty}(\mathcal{F}_{t_1})$ there exist $y,z\in\R$
such that $y\preceq_{0,1} f$ and $x\succeq_{0,1} f$.
\end{enumerate}
\end{description}

The Axiom (T.0) is composed by four
requirements only and the reason of this significant simplification is due to the assumption $\Fcal_0=\{\emptyset,\Omega\}$. To understand why completeness and transitivity are the natural counterpart suggested by the classical definition of weak order, we observe that Proposition \ref{representation:0} 
guarantees that under (T.0) for any $f \in
\cl^{\infty}(\mathcal{F}_{t_1})$ there exists a unique $C_{0,1}(f)\in \R$ such that both $C_{0,1}(f) \succeq_{0,1}
f$ or $C_{0,1}(f) \preceq_{0,1} f$ hold.
We can therefore consider the following induced ordering $\preceq_1$: for any $f,g \in
\cl^{\infty}(\mathcal{F}_{t_1})$, $f\preceq_1 g$ if and only if $C_{0,1}(f)\leq C_{0,1}(g)$. Indeed $\preceq_1$ is reflexive and inherits completeness and transitivity from $\preceq_{0,1}$.

\medskip

We now move to the interpretation of Axiom (T.i): properties 1, 5 and 6 are deeply related and inspired to the
notion of conditional preference in \cite{DJ14}. The first property of (T.i) points out that the updating procedure necessarily leads to preferences which are complete in a conditional sense.  In particular we shall see in Lemma \ref{treeventi} (which is the counterpart of Lemma 3.2 in \cite{DJ14}) that local completeness allows to
compare two acts on three disjoint $\Fcal_{t_i}$ measurable
events. Consistency and stability can be understood in terms of
information achieved: for example consistency states that if the
agent prefer $g\in \cl^{\infty}(\mathcal{F}_{t_i})$ at
time $t_i$ rather than $f \in
\cl^{\infty}(\mathcal{F}_{t_{i+1}})$ at time $t_{i+1}$
knowing that event $A \in \Fcal_{t_i}$ has occurred, than she
shall prefer $g$ for any condition $B\in \Fcal_{t_i}$, $B\subseteq A$.  
\\ The property 2 in (T.i) is the conditional generalization of its counterpart in (T.0). Normalization (property 3) says that $\Fcal_{t_i}$ null events are preserved in the one step updating.   
In particular the agent is indifferent between random payoffs which differ from $0$ by a negligible $\Fcal_{t_i}$ measurable set (Loosely speaking \textquotedblleft Holding nothing is indifferent throughout time\textquotedblright, up to null events). Non degeneracy (property 4) is the more technical one and guarantees some
simplifications in our arguments, since it implies that any
random payoff $f$ at time $t_{i+1}$ admits an $\Fcal_{t_i}$-measurable $g$  which is more/less preferred (it is nevertheless a very weak requirement which is satisfied in all the cases of interest).

\begin{example}
In the Introduction (pp. 3-4) we proposed the classical framework of utility maximization, which can help understanding the meaning of Axiom (T.i) 1,5,6. 
In fact the preference relation $X \succeq_{t,T} V_T(t,X,\alpha)$ is defined via the inequality between $\Fcal_t$-measurable 
random variables, $v(t,X)\geq E_{\Prob}[u(V_T(t,X,\alpha))\mid \Fcal_t]$, and inherits those properties which characterize 
the conditional expectation (namely Axiom (T.i) 1,5,6, replacing $t_i$ with $t$ and $t_{i+1}$ with $T$).     
\end{example}

The definition of Conditional Certainty Equivalent is the basis of
our representation results and  follows from the idea in
\cite{FM11}. In Section \ref{inductive:proof} we shall show by
induction the existence (and uniqueness) of the Conditional
Certainty Equivalent at each time step.

\begin{definition}\label{CCEi} We say $g\sim_{i,i+1} f$ if and only if $g \succeq_{i,i+1} f$
and $g \preceq_{i,i+1} f$. If $g\sim_{i,i+1} f$ then we shall call
$g$ the Conditional Certainty Equivalent (CCE) of $f$ and denote the
family of all CCEs as $C_{i,i+1}(f)$.
\end{definition}

\begin{notation}\label{notation} In what follows we shall use these notations for any $g \in
\cl^{\infty}(\mathcal{F}_{t_i})$ and $f \in
\cl^{\infty}(\mathcal{F}_{t_{i+1}})$:
\begin{itemize}

\item $g \sim_{i,i+1} f$ if both $g \succeq_{i,i+1} f$ and $g
\preceq_{i,i+1} f$ hold;

\item $g \succ_{i,i+1} f$ if $g \succeq_{i,i+1} f$ but $g
\mathbf{1}_A \not\sim_{i,i+1} f \mathbf{1}_A$ $\forall A \in
\mathcal{F}_{t_i} \setminus \mathcal{N}(\mathcal{F}_{t_i})$;

\item $g \prec_{i,i+1} f$ if $g \preceq_{i,i+1} f$ but $g
\mathbf{1}_A \not\sim_{i,i+1} f \mathbf{1}_A$ $\forall A \in
\mathcal{F}_{t_i}\setminus \mathcal{N}(\mathcal{F}_{t_i})$;

\item $g \succ_{i,i+1}^{A} f$ if $g\mathbf{1}_A \succeq_{i,i+1} f
\mathbf{1}_A$ but $g \mathbf{1}_B \not\sim_{i,i+1} f \mathbf{1}_B$
$\forall B \in \mathcal{F}_{t_i} \setminus
\mathcal{N}(\mathcal{F}_{t_i})$ with $B\subseteq A$;

\item $g \prec_{i,i+1}^{A} f$ if $g\mathbf{1}_A \preceq_{i,i+1} f
\mathbf{1}_A$ but $g \mathbf{1}_B \not\sim_{i,i+1} f \mathbf{1}_B$
$\forall B \in \mathcal{F}_{t_i} \setminus
\mathcal{N}(\mathcal{F}_{t_i})$ with $B\subseteq A$.
\end{itemize}
\end{notation}

\begin{remark}\label{regularity} We observe that consistency jointly to stability of $\succeq_{i,i+1}$ (similar for $\preceq_{i,i+1}$)
imply the following pasting properties:
\begin{itemize}
 \item for any $A,B\in\Fcal_{t_i}$, $g_1,g_2\in
 \cl^{\infty}(\mathcal{F}_{t_i})$ and
 $f_1,f_2\in\cl^{\infty}(\mathcal{F}_{t_{i+1}})$.
If $g_1 \mathbf{1}_A \succeq_{i,i+1} f_1
 \mathbf{1}_A$ and  $g_2 \mathbf{1}_{B} \succeq_{i,i+1} f_2
 \mathbf{1}_{B}$ then  
 $(g_1+ g_2)\mathbf{1}_{A \cap B} \succeq_{i,i+1} (f_1 +f_2)\mathbf{1}_{A\cap B}$, 
 $g_1\mathbf{1}_{A \setminus B} \succeq_{i,i+1} f_1 \mathbf{1}_{A\setminus B}$ and 
 $g_2\mathbf{1}_{B \setminus A} \succeq_{i,i+1} f_2\mathbf{1}_{B\setminus A}$.

\item for a family $\{A_n\}_{n\in\N}\subseteq \Fcal_{t_i}$ of disjoint events and
$A=\cup_n A_n$ we have
$$ g\mathbf{1}_A \preceq_{i,i+1} f\mathbf{1}_A \;\Leftrightarrow\; g\mathbf{1}_{A_n} \preceq_{i,i+1} f\mathbf{1}_{A_n} \text{ for every } n.$$
\end{itemize}
\end{remark}

Axiom (T.i) is the key ingredient to obtain the updating construction we are aiming at. In particular assume that at time $t_i$ the agent is characterized by a couple $(\Prob_i,u_i)$ where $\Prob_i$ is the subjective probability on measurable events $\Fcal_{t_i}$ and $u_i$ is a state dependent utility such that $u_i(x,\cdot)$ is $\Fcal_{t_i}$-measurable. 
We shall prove in Proposition \ref{induction:assumption} that if $\succeq_{i,i+1}$ satisfies (T.i) then for any $f\in
\cl^{\infty}(\mathcal{F}_{t_{i+1}})$ there exists a unique Conditional Certainty Equivalent given by
$C_{i,i+1}(f)=u_i^{-1}V_{i+1}(f)$, where
$V_{i+1}(f)  =  \Prob_i-\inf\{u_i(g) \mid g\succeq_{i,i+1} f \}$. Moreover $V_{i+1}$
represents the transition order i.e.
\begin{eqnarray*}
g\preceq_{i,i+1} f & \Leftrightarrow & u_i(g)\leq V_{i+1}(f) \quad
\Prob_i\text{-a.s.}
\\ g\succeq_{i,i+1} f & \Leftrightarrow & u_i(g)\geq V_{i+1}(f)
\quad \Prob_i\text{-a.s.} 
\end{eqnarray*}

\paragraph{Integral representation of Inter Temporal Preferences.} We will take into consideration the following axioms: 
monotonicity, the Sure Thing Principle and a technical continuity,
adapted to this conditional setting, which will lead to a
representation of the ITP in the desired integral form.

\begin{description}\item[(M.i)] Strict Monotonicity.
Given arbitrary $g_1, g_2, g_3, \in
\cl^{\infty}(\mathcal{F}_{t_{i}})$, $f\in \cl^{\infty}(\mathcal{F}_{t_{i+1}})$, $A\in
\Fcal_{t_{i+1}}\setminus \mathcal{N}(\mathcal{F}_{t_{i+1}})$ and
$g_1 < g_2$:
$$g_3 \sim_{i,i+1} g_1 \mathbf{1}_A + f \mathbf{1}_{A^c} \text{
implies } g_3 \prec_{i,i+1}^B g_2 \mathbf{1}_A + f
\mathbf{1}_{A^c} \text{ for some } B\in \Fcal_{t_i}\setminus
\mathcal{N}(\mathcal{F}_{t_{i}}),$$
$$ g_3 \sim_{i,i+1} g_2
\mathbf{1}_A + f \mathbf{1}_{A^c} \text{ implies }  g_3
\succ_{i,i+1}^B g_1 \mathbf{1}_A + f \mathbf{1}_{A^c}  \text{ for
some } B\in \Fcal_{t_i}\setminus
\mathcal{N}(\mathcal{F}_{t_{i}}).$$

\item[(ST.i)] Sure-Thing Principle. Given arbitrary $f_1,f_2,h\in
\mathcal{S}_{\mathcal{F}_{t_{i}}}(\mathcal{F}_{t_{i+1}})$, $A\in
\Fcal_{t_{i+1}}\setminus\mathcal{N}(\mathcal{F}_{t_{i+1}})$ and
$g_1\in \cl^{\infty}(\mathcal{F}_{t_{i}})$, such that $g_1
\succeq_{i,i+1} f_1\mathbf{1}_A+h \mathbf{1}_{A^c} $ and $g_1
\preceq_{i,i+1} f_2\mathbf{1}_A+h \mathbf{1}_{A^c}$: for any $k\in
\mathcal{S}_{\mathcal{F}_{t_{i}}}(\mathcal{F}_{t_{i+1}})$ there
exists $g_2\in \cl^{\infty}(\mathcal{F}_{t_{i}})$ such
that $g_2 \succeq_{i,i+1} f_1\mathbf{1}_A+k \mathbf{1}_{A^c} $ and
$g_2 \preceq_{i,i+1} f_2\mathbf{1}_A+k \mathbf{1}_{A^c}$.

\item[(C.i)] Pointwise continuity. Consider any uniformly bounded
sequence $\{f_n\}\subseteq
\cl^{\infty}(\mathcal{F}_{t_{i+1}})$, such that
$f_n(\omega)\rightarrow f(\omega)$ for any $\omega\in\Omega$, then
for any $g\prec_{i,i+1}f$ (resp. $g\succ_{i,i+1}f$ ) there exists
a partition $\{A_k\}_{k=1}^{\infty}\subset \Fcal_{t_i}$ such that
for any $k$ we have
$g\mathbf{1}_{A_k}\preceq_{i,i+1}f_{n}\mathbf{1}_{A_k}$ (resp.
$g\mathbf{1}_{A_k}\succeq_{i,i+1}f_{n}\mathbf{1}_{A_k}$ ) for all
$n\geq n_k$.

\end{description}


We are now ready to state the main contribution of this paper:
Theorem \ref{main:theorem} provides the representation of ITP in
terms of a unique probability $\Prob$ and a stochastic field
$u(t,x,\omega)$, which describes the random fluctuations of
preferences. These were exactly the elements exploited in \cite{FM11} to determine the dynamics of the Conditional
Certainty Equivalent.

\begin{theorem}[Representation]\label{main:theorem} Let Assumption
\ref{overall:assumption} holds and any $\Fcal_{t_i}$ contains
three essential disjoint events for every $i=1,2,\dots$. The
intertemporal preference $\succeq_{i,i+1}$ satisfies (T.i),
(M.i), (ST.i) and (C.i) for any $i=0,\dots, N$ if and only if
there exist a probability $\Prob$ on $\Fcal_{t_N}$ and a
Stochastic Dynamic Utility

\begin{eqnarray}\label{SDU}u(t,x,\omega)=\sum_{i=0}^{N-1}
u_i(x,\omega)\mathbf{1}_{[t_i,t_{i+1})}(t)+u_N(x,\omega)\mathbf{1}_{t_N}(t)
\end{eqnarray}
satisfying

\begin{enumerate}

\item[(a)] $u(t_i,x,\cdot)$ is $\Fcal_{t_i}$-measurable and  $E_{\Prob}[|u(t_i,x,\cdot)|]<\infty$, for all
$x\in\R$;

\item[(b)] $u(t_i,\cdot,\omega)$ is strictly increasing in $x$ and
$u(t_i,0,\omega)=0$\footnote{This additional requirement is in
fact without loss of generality, and allows a useful
simplification in the main body of the proof.}, for all $\omega\in
\Omega$;

\item[(c)] $u(t_i,\cdot,\cdot)$ is $\star$-continuous 

\item[(d)] $E_{\mathbb{P}}[u_{i+1}(f)| \mathcal{F}_{t_i}]\in L(\Fcal_{t_i}; u_i)$\footnote{See the definition in Equation \eqref{range}} for any $f\in \cl^{\infty}(\Fcal_{t_{i+1}})$, $g\in \cl^{\infty}(\Fcal_{t_i})$ and
\begin{eqnarray*}g \succeq_{i,i+1} f & \iff & u(t_i, g) \ge
E_{\mathbb{P}}[u(t_{i+1},f)| \mathcal{F}_{t_i}] \quad
\mathbb{P}\text{-a.s.}
\\ g \preceq_{i,i+1} f & \iff & u(t_i, g) \le
E_{\mathbb{P}}[u(t_{i+1},f)| \mathcal{F}_{t_i}] \quad
\mathbb{P}\text{-a.s.}
\end{eqnarray*}
\end{enumerate}
\emph{Relative uniqueness:} the couple $(\Prob,u)$ can be replaced
by $(\mathbb{P}^*,u^*)$ if and only if $\mathbb{P}$ is equivalent
to $\mathbb{P}^*$ on $\Fcal_{t_N}$ and for any $i=1,\dots,N$ we
have $\Prob(u^*(t_i,\cdot,\cdot) = \delta_i u_{i})=1$, where
$\delta_i$ is the Radon-Nikodym derivative of
$\mathbb{P}_{|\Fcal_{t_i}}$ with respect to
$\mathbb{P}^*_{|\Fcal_{t_i}}$.
\end{theorem}

\begin{example}[Forward performances]
\label{comparison} Comparing the representation of ITP provided in Theorem \ref{main:theorem} with the existing literature about forward performances (see for instance \cite{AZZ18,MZ06,MZ09}), we may immediately notice that our approach does not rely on the existence of a financial market.  We recall that an adapted
process $U(x,t)$ on a fixed probability space $(\Omega,\Fcal, \{\Fcal_t\}_{t\geq 0}, \Prob)$ is said to be a Forward Performance if: i) it is increasing and concave as a function of $x$ for each $t$; ii) $U(x,0)=u_0(x)\in \R$; iii) for all $T\geq t$ and each self-financing strategy represented by $\pi
$, the associated discounted wealth $X^{\pi}$ satisfies
$E_{\Prob}[U(X_T^{\pi},T)\,|\,\Fcal_t]\leq U(X_t^{\pi},t)$;
iv) for all $T\geq t$ there exists a self-financing strategy $\pi^*$ such
that $X^{\pi^*}$ satisfies the equality in point iii).
\\ A posteriori we therefore know that the couple $(\Prob, U(t,x))$ defines an intertemporal relation $\succeq_{s,t}$ as usual by $U(s, \cdot) \ge
E_{\mathbb{P}}[U(t,\cdot)| \mathcal{F}_{s}] \quad
\mathbb{P}\text{-a.s.}$. In partiular for the optimal policy we have the relation $X_s^{\pi^*}\sim_{s,t} X_t^{\pi^*}$\footnote{This can be compared with the discussion in the Introduction, see Equation \eqref{Merton}.}. On the other hand not all the intertemporal preferences are necessarily related to an existing financial market.  
\end{example}

\paragraph{On the discount factor.} The role of discounting in the theory of dynamic choices can be described on two different layers. 
\\\emph{First layer:} in equations \eqref{consumption1}, \eqref{consumption2} we see an explicit dependence on a discount factor related to the utility $u$, which is motivated by the fact that $u$ is homogeneous in time. Indeed in our framework  $u(t,x,\omega)$ varies stochastically in time and therefore it is not possible to disentagle the contribution of discounting from the utility in a unique way. Moreover the uniqueness of the representation is up to equivalent change of measures and therefore the discount factor would be in any case sensitive to probabilistic measure changes. Nevertheless there are situations in which it is possible to determine a discounting process. For instance assume the Decion Maker finds a couple $\Prob, u$ which represents the ITP as in Theorem \ref{main:theorem}, (d). At the same time the Decion Maker may be endowed by a subjective probability $\Prob^*$ and decide to perform a change of measure   without changing the Stochastic Dynamic Utility $u$. In such a situation the adapted process $\{\beta_t\}_{t\geq 0}$ defined by $\beta_t=E_{\Prob}[\frac{d\Prob}{d\Prob^*}|\Fcal_t]$ can be interpreted as a stochastic discount factor: indeed for $g\in \cl^{\infty}(\Omega,\Fcal_{t_i})$, $f\in \cl^{\infty}(\Omega,\Fcal_{t_{i+1}})$ 
\begin{eqnarray*}g \succeq_{i,i+1} f & \iff & \beta_{t_i}u(t_i, g) \ge
E_{\mathbb{P}^*}[\beta_{t_{i+1}}u(t_{i+1},f)| \mathcal{F}_{t_i}] \quad
\mathbb{P}^*\text{-a.s.}
\\ g \preceq_{i,i+1} f & \iff & \beta_{t_i} u(t_i, g) \le
E_{\mathbb{P}^*}[\beta_{t_{i+1}}u(t_{i+1},f)| \mathcal{F}_{t_i}] \quad
\mathbb{P}^*\text{-a.s.}
\end{eqnarray*} 
\\\emph{Second layer:} in the theory of Forward Performances \cite{MZ06,MZ09}, the representation of ITP does not show any explicit dependence on a (stochastic) discount factor, but the utility $U$ is computed directly on discounted wealth processes assuming that a reference numeraire exists \emph{a priori}. Let $\{B_t\}_{t\in [0,+\infty)}$ be an adapted stochastic process acting as a numeraire, with $\Prob(B_t>\varepsilon)=1$ for some $\varepsilon>0$. If $\Prob, u$  represent the ITP then we can set $u^*(t,x,\omega):=u(t,x\cdot B_t(\omega),\omega)$ and obtain 
\begin{eqnarray*}g \succeq_{i,i+1} f & \iff & u^*(t_i, g^*) \ge
E_{\mathbb{P}}[u^*(t_{i+1},f^*)| \mathcal{F}_{t_i}] \quad
\mathbb{P}\text{-a.s.}
\\ g \preceq_{i,i+1} f & \iff &  u^*(t_i, g^*) \le
E_{\mathbb{P}}[u^*(t_{i+1},f^*)| \mathcal{F}_{t_i}] \quad
\mathbb{P}\text{-a.s.}
\end{eqnarray*}
where $g^*=\frac{g}{B_{t_i}}$ and $f^*=\frac{f}{B_{t_{i+1}}}$ are the discounted values of $g$ and $f$.

\paragraph{Time consistency of intertemporal preferences.}
The family $\{\succeq_{i,i+1}\}$ of intertemporal preferences is meant to create a link between two successive times $t_i$ and $t_{i+1}$ in order to compare random payoffs whose effects will be known and exploitable at different times. The procedure is a step by step updating and simple inspections show that the following semigroup property holds true for the Conditional Certainty Equivalent
\begin{equation}\label{semigroup}
C_{s,v}(f)= C_{s,t}(C_{t,v}(f)) 
\footnote{Abuse of notation: the precise formulation should be $C_{s,v}(f)= C_{s,t}(g)$ where $g\in \cl(\Fcal_t)$ is a version of $C_{t,v}(f)$.}
\quad \forall\,0\leq s<t<v \text{ and } f\in \cl^{\infty}(\Omega,\Fcal_{v})  
\end{equation}
where for any $s<t$ the operator $C_{s,t}(\cdot)$ is the ($\Prob$-a.s. unique) solution of the equation $u(s,C_{s,t}(\cdot))=E_{\Prob}[u(t,\cdot)\mid \Fcal_s]$ and $u$ is the Stochastic Dynamic Utility obtained in Theorem \ref{main:theorem}. As an immediate consequence  we can extend the intertemporal preferences to any $s<t$ as follows 
\begin{eqnarray*}g \succeq_{s,t} f & \iff & u(s, g) \ge
E_{\mathbb{P}}[u(t,f)| \mathcal{F}_{s}] \quad
\mathbb{P}\text{-a.s.}
\\ g \preceq_{s,t} f & \iff & u(s, g) \le
E_{\mathbb{P}}[u(t,f)| \mathcal{F}_{s}] \quad
\mathbb{P}\text{-a.s.}
\end{eqnarray*}
where $g\in \cl^{\infty}(\Omega,\Fcal_{s})$ and $f\in \cl^{\infty}(\Omega,\Fcal_{t})$. In virtue of the semigroup property \eqref{semigroup} we obtain the following time consistency of preferences
\begin{proposition}\label{timecon} Let $0\leq s<t<v$ and let $g\in \cl^{\infty}(\Omega,\Fcal_{s})$ and $f\in \cl^{\infty}(\Omega,\Fcal_{v})$ such that 
 $g \succeq_{s,v} f$ (resp. $g \preceq_{s,v} f$).  Then $g\succeq_{s,t} h$ (resp. $g \preceq_{s,t} h$) for any $h\in \cl^{\infty}(\Omega,\Fcal_{t})$ such that $h\sim_{t,v} f$.     
\end{proposition}

\medskip

Since the proof of Theorem \ref{main:theorem} will proceed
inductively we choose to present the theory in the simpler
unconditioned case $\preceq_{0,1}$ (see Section
\ref{unconditioned:updating}). The results in the next section
will be therefore necessary to prove the initial step in the
induction argument of Theorem \ref{main:theorem}.
\\Moreover we stress that the relative uniqueness is sharper than in representation results like those contained in \cite{CL06,WZ99}. This
follows from the fact that the $u_0$ is fixed a priori (together with the normalization condition $u(t_i,0,\omega)=0$) 
and plays the role of an initial (constraining) condition (see also Proposition \ref{main:0} for further details).


\section{Unconditioned intertemporal preference}\label{unconditioned:updating}


We consider a Decision Maker  who compares an initial amount of
some good, whose value is surely determined (and its benefit is
immediate) with respect to bounded random payoffs (e.g. bets,
assets, future value of goods) at a fixed time $t_1$ represented
by elements in the space
$\cl^{\infty}(\mathcal{F}_{t_1})$. We say that the agent
is initially naive, as the initial information are represented by
the trivial $\Fcal_0=\{\emptyset,\Omega\}$ and therefore the space
$\cl^{\infty}(\mathcal{F}_0)$ is isometric to the real
line $\mathbb{R}$.

Consider the transition preference $\preceq_{0,1}$ (or
$\succeq_{0,1}$) which connects
$\cl^{\infty}(\mathcal{F}_{t_1})$ to
$\cl^{\infty}(\mathcal{F}_{0})$. As already observed in the case $i=0$ the first
Axiom (T.0) is composed only by four requirements: completeness, transitivity, normalization and non
degeneracy (which is the more technical requirement we shall use in Lemma \ref{equality0}).

\begin{remark} From Notation \ref{notation} we can easily deduce the meaning of the symbols $\sim_{0,1}$,
$\succ_{0,1}$, $\prec_{0,1}$. We also
recall that the set of null events induced by $\preceq_{0,1}$ is
given by
\begin{equation*}\label{Null1}\mathcal{N}(\mathcal{F}_{t_1}) = \{ A \in \mathcal{F}_{t_{1}}
: a \sim_{0,1} f \Rightarrow a \sim_{0,1} b \mathbf{1}_A + f
\mathbf{1}_{A^c}, \forall f \in
\cl^{\infty}(\mathcal{F}_{t_{1}}), a,b \in \R \}.
\end{equation*}
\end{remark}

\begin{definition}\label{CCE0} If $a\sim_{0,1} f$ then we shall call $a$
the (Conditional) Certainty Equivalent of $f$ and denote the
family of all CCEs as $C_{0,1}(f)$.
\end{definition}

We now show that under (T.0) the CCE exists and is unique. Notice
that this notion of certainty equivalent matches the dynamic
generalization introduced by \cite{FM11}. The CCE will also
provide a natural representation of the intertemporal preference
$\preceq_{0,1}$ (see the following Proposition
\ref{representation:0}).
\\Consider the maps
\begin{eqnarray*} V_{1}^-(f)  =  \sup\{u_0(a) \mid a\preceq_{0,1} f  \}
&\text{ and }& V_{1}^+(f)  =  \inf\{u_0(a) \mid a\succeq_{0,1} f
\},
\end{eqnarray*}
where $u_0$ will be always supposed to fullfill Assumption \ref{overall:assumption}. 

We note that in the definition of $V_{1}^{\pm}(f)$, $u_0$ needs not to be fixed. Indeed if we consider the total ordering on $\cl^{\infty}(\mathcal{F}_{t_1})$ induced by the functionals $V_{1}^{\pm}(\cdot)$ (i.e. $f_1\preceq_1 f_2$ if and only if $V_{1}^{\pm}(f_1)\leq V_{1}^{\pm}(f_2)$) this would not be affected by the choice of $u_0$. Nevertheless as previously explained  we prefer to think $u_0$ as an initial data characterizing the decision maker, with the advantage of obtaining a sharper notion of uniqueness.     

\begin{lemma}\label{equality0} Under (T.0) and Assumption \ref{overall:assumption} the maps $V_1^+, V^-_1$ are well defined from
$\cl^{\infty}(\mathcal{F}_{t_1})$ to $\R$. Moreover $V_1^+(f)=V^-_1(f)$ for any $f\in
\cl^{\infty}(\mathcal{F}_{t_1})$.
\end{lemma}

\begin{proof} From completeness $V_1^{\pm}$ are well
defined and taking values in $\R\cup \{\pm \infty\}$. The fact
that $V_1^{\pm}(f)$ are finite for any $f\in
\cl^{\infty}(\mathcal{F}_{t_1})$ follows from non
degeneracy in Axiom (T.0). \\For any $a,b\in\R$ such that $a
\preceq_{0,1} f$ and $b \succeq_{0,1} f$ we have $u_0(a)\leq
u_0(b)$ and therefore $V_1^-(f)\leq V^+_1(f)$. Now assume by
contradiction $V_1^-(f)< V^+_1(f)$: since $u_0$ is strictly
increasing and continuous there exists $c$ such that $u_0(c)\in
(V_1^-(f),V^+_1(f))$. From completeness either $c \succeq_{0,1} f$
or $c \preceq_{0,1} f$ getting in both cases a contradiction since
$$\sup\{u_0(a) \mid a\preceq_{0,1} f  \}
 < u_0(c)< \inf\{u_0(a) \mid a\succeq_{0,1} f
\}.$$
\end{proof}

\begin{notation} From now on, whenever (T.0) and Assumption \ref{overall:assumption} are in force we shall denote $V_1 := V_1^+ \equiv V^-_1$.
\end{notation}

\begin{proposition} \label{representation:0} Let (T.0) and  Assumption \ref{overall:assumption} hold. Then for any $f\in \cl^{\infty}(\mathcal{F}_{t_{1}})$ there
exists a unique Conditional Certainty Equivalent given by
$C_{0,1}(f)=u_0^{-1}V_1(f)$. Moreover $V_1$ takes values in the range of $u_0$ and represents the
transition order i.e.
\begin{eqnarray}
a\preceq_{0,1} f & \Leftrightarrow & u_0(a)\leq V_1(f) \label{representation1}
\\ a\succeq_{0,1} f & \Leftrightarrow & u_0(a)\geq V_1(f) \label{representation2}
\end{eqnarray}
\end{proposition}

\begin{proof} Existence and uniqueness follow from the previous Lemma
\ref{equality0} and Assumption \ref{overall:assumption}. Notice
that $a\preceq_{0,1} f$ (resp. $a\succeq_{0,1} f$) obviously
implies $ u_0(a)\leq V_1(f)$ (resp. $u_0(a)\geq V_1(f)$). For the
reverse implication we can observe that $u_0(a)=V_1(f)$ implies
$a\sim_{0,1} f$. If instead $u_0(a)<V_1(f)$ (resp.
$u_0(a)>V_1(f)$), then necessary $a\prec_{0,1} f$ (resp.
$a\succ_{0,1} f$) as $V_1(f)=\inf\{u_0(a) \mid a\succeq_{0,1} f
\}$ (resp. $V_1(f)=\sup\{u_0(a) \mid a\preceq_{0,1} f \}$).
\end{proof}

We will take into consideration the following axioms: 
monotonicity, the Sure Thing Principle and a technical continuity, which we recall here to clarify their meaning in this simplified unconditioned case.

\begin{description}\item[(M.0)] Strict Monotonicity:
for all $a, b, c \in \mathbb{R}$, $f\in \cl^{\infty}(\mathcal{F}_{t_{1}})$ $A\in
\Fcal_{t_1}\setminus\mathcal{N}(\mathcal{F}_{t_1})$ and $a < b$ we
have $c\sim_{0,1} a\mathbf{1}_A + f \mathbf{1}_{A^c}$ implies $c
\prec_{0,1} b \mathbf{1}_A + f \mathbf{1}_{A^c}$ (resp.
$c\sim_{0,1} b \mathbf{1}_A + f \mathbf{1}_{A^c}$ implies $c
\succ_{0,1} a \mathbf{1}_A + f \mathbf{1}_{A^c}$).


\item[(ST.0)] Sure-Thing Principle: consider arbitrary $f,g,h\in
\mathcal{S}(\mathcal{F}_{t_1})$, $A\in
\Fcal_{t_1}\setminus\mathcal{N}(\mathcal{F}_{t_1})$ and $a\in \R$
such that $a \succeq_{0,1} f\mathbf{1}_A+h \mathbf{1}_{A^c} $ and
$a \preceq_{0,1} g\mathbf{1}_A+h \mathbf{1}_{A^c}$ then for any
$k\in \mathcal{S}(\mathcal{F}_{t_1})$ there exists $b\in\R$ such
that $b \succeq_{0,1} f\mathbf{1}_A+k \mathbf{1}_{A^c} $ and $b
\preceq_{0,1} g\mathbf{1}_A+k \mathbf{1}_{A^c}$.

\item[(C.0)] Pointwise continuity: consider any uniformly bounded
sequence $\{f_n\}\subseteq \FactI$, such that
$f_n(\omega)\rightarrow f(\omega)$ for any $\omega\in\Omega$, then
for all $a\prec_{0,1}f$ (resp. $a\succ_{0,1}f$ ) there exists $N$
such that $a\preceq_{0,1}f_n$ (resp. $a\succeq_{0,1}f_n$ ) for $n>
N$.

\end{description}

\begin{remark} In the classical Decision Theory (see \cite{Pe16})
the Sure-Thing Principle is a sort of independence principle: it
says that the preference between two acts, $f$ and $g$, should
only depend on the values of $f$ and $g$ when they differ. If $f$
and $g$ differ only on an event $A$, if $A$ does not occur $f$ and
$g$ result in the same outcome exactly. In our intertemporal
framework the interpretation is exactly the same, even though we
need to deal with the comparison at time $0$. 
\end{remark}

\begin{remark}
In the present context the Sure-Thing Principle (ST.0) easily implies
for arbitrary $f,g\in \cl^{\infty}(\mathcal{F}_{t_1})$ and
$A\in \mathcal{F}_{t_1}$: $V_1(f\mathbf{1}_A)\leq V_1(g
\mathbf{1}_A)$ and $V_1(f\mathbf{1}_{A^c})\leq V_1(g
\mathbf{1}_{A^c})$ then $V_1(f)\leq V_1(g)$.
\end{remark}

In the remaining of this section we shall prove the following

\begin{proposition}\label{main:0} Assume that $\Fcal_{t_1}$
contains at least three disjoint essential events and Assumption \ref{overall:assumption} is in force. Axioms (T.0), (M.0), (ST.0) and (C.0) hold if and only if there exists a probability $\mathbb{P}_1$ on
$\Omega$ and a function $u_1(\cdot, \omega) : \mathbb{R} \to
\mathbb{R}$, strictly increasing $\forall \omega \in \Omega$ and $\star$-continuous  such that the functional $V_1$ 
\begin{equation}\label{repr.Step1}
V_1(f)= \int_{\Omega} u_1(f(\omega),\omega) d\mathbb{P}_1
\end{equation}
represents the preference $\preccurlyeq_{0,1}$ (in the sense of \eqref{representation1} and \eqref{representation2}) and takes values in the range of $u_0$.
\\The following uniqueness holds for \eqref{repr.Step1} :
$(\mathbb{P}_1,u_1)$ can be replaced by $(\mathbb{P}^*,u^*)$ if
and only if $\mathbb{P}_1$ is equivalent to $\mathbb{P}^*$ and
$\Prob_1\left( u^* = \delta u_1+\tau\right)=1,$ where $\delta$ is
the Radon-Nikodym derivative of $\mathbb{P}_1$ with respect to
$\mathbb{P}^*$ and $\tau\in \cl(\mathcal{F}_{t_1})$ with
$E_{\mathbb{P}^*}[\tau]=0$.
\end{proposition}

\begin{remark}\label{remark:u0} Even though Proposition \ref{main:0} shows many similarities with Theorem \ref{formaintegrale}, some work needs to be done to show that Axioms (T.0), (M.0), (ST.0) and (C.0) are sufficient to apply the results in \cite{CL06,WZ99}. Moreover as $V_1$ is defined via a fixed $u_0$, we shall show that the coefficient $\sigma>0$ appearing in Theorem \ref{formaintegrale} is necessarily equal to $1$. 
\end{remark}

\begin{remark}\label{integrable} We observe that even if not mentioned explicitly, necessarily the random variable $u_1(x, \cdot)$
is integrable with respect to $\Prob_1$ for any $x\in\R$.
\\Moreover if we impose the normalization requirement $u_1(0,\omega)=0$ for every
$\omega\in\Omega$ then $\tau$  is equal to $0$ $\Prob$-a.s..
\end{remark}

\subsection{Proof of Proposition \ref{main:0}}

Observe that the hypothesis of Proposition \ref{representation:0} are satisfied. Hence, the representation $a\succcurlyeq_{0,1} f \iff u_0(a)\geq V_1(f)$ holds where $V_1$ is defined as in Lemma \ref{equality0}. Furthermore, for any $f\in \cl^\infty(\cf_{t_1})$ the CCE $C_{0,1}(f)$ exists and is uniquely given by $u_0^{-1}(V_1(f))$. The existence of the CCE for every act $f$ directly implies that the range of the function $V_1$ is contained in the range of $u_0$.

\paragraph{Proof of ($\Rightarrow$)} We define a weak order on $\FactI$ as $f\preceq
g$ if and only if $V_1(f)\leq V_1(g)$. (T.0) implies $\preceq$ is
complete, reflexive and transitive (i.e. satisfies (A1) in the
Appendix). \\Let $f \in \FactI$ and outcomes $x > y$: indeed (M.0)
implies $V_1(x\mathbf{1}_A + f\mathbf{1}_{A^c})> V_1(
y\mathbf{1}_A + f\mathbf{1}_{A^c})$, for all nonnull events $A\in
\Fcal_{t_1}$ and hence $\preceq$ is strictly monotone in the sense
of (A2). Similarly (ST.0) implies that $\preceq$ satisfies (A3).
\\ Let now $\{f_n\}\subseteq \FactI$, such that
$f_n(\omega)\rightarrow f(\omega)$ for any $\omega\in\Omega$ and
$\|f_n\|_{\infty}<k$ for all $n\in \N$. Let now $g\in \FactI$ such that $g \succ f$ and consider $a=C_{0,1}(g)$ (which exists by Proposition \ref{representation:0}). Then $a\succ_{0,1} f$ and by (C.0) we can find $\bar{n}$ such that for all $n\ge \bar{n}$ we have $a\succ_{0,1} f_n$. Therefore $V_1(g)=u_0(a)>V_1(f_n)$ (similarly for the opposite inequality) showing that (A4) holds for $\preceq$ .


\medskip

We can therefore apply Theorem \ref{formaintegrale} and find the desired
representation \eqref{repr.Step1} namely $V_1(f)= \int_{\Omega}
u_1(f(\omega),\omega) d\mathbb{P}_1 = E_{\Prob_1}[u_1(f)]$ and its uniqueness. Let therefore $\Prob^*$ and $u^*(\cdot,
\cdot) = \tau + \sigma\delta u_1(\cdot, \cdot)$ obtained by
Theorem \ref{formaintegrale}. Observe that $V_1(0)=u_0(0)=0$
implies $E_{\mathbb{P}^*}[\tau]=0$. Moreover as
$u_0(C_{0,1}(f))=V_1(f)=E_{\Prob^*}[u^*(f)]=E_{\Prob_1}[\sigma
u_1(f)]$, we have necessarily $\sigma=1$.

\medskip

We now show that the state dependent utility $u_1$ is $\star$-continuous on $(\Omega,\Fcal_{t_1},\Prob_1)$.
\\To this end consider any $f\in \cl^\infty(\Fcal_{t_1})$. It is sufficient to show that $\Prob_1(LD_f)=0$ where $LD_f$ is the set defined in Appendix \ref{Appendix A} replacing $\phi$ with $u_1$. Indeed, with an analogous argument one obtains $\Prob_1(RD_f)=0$. Ultimately, the thesis follows from the observation that $\Prob_1(D_f)=\Prob_1(LD_f\cup RD_f) = 0$ and the arbitrariness of $f$.

\noindent As consequence of Lemma \ref{lma:measurable sets}, $LD_f\in \Fcal_{t_1}$, hence either $\Prob_1(LD_f)=0$ or $\Prob_1(LD_f)>0$. Suppose, by contradiction, that there exists $f^* \in \cl^\infty(\Fcal_{t_1})$ such that $\Prob_1(LD_{f^*}) > 0$. Set $B:= LD_{f^*}$ and let $f = f^*\mathbf{1}_{B}$ and $f_n = (f - \frac{1}{n})\mathbf{1}_B$. By construction $f,f_n \in \cl^\infty(\Fcal_{t_1})$ for each $n\in\mathbb{N}$, $f_n(\omega) \to f(\omega)$ for each $\omega \in \Omega$ and $\sup_n \norm{f_n}_\infty \leq \norm{f}_\infty + 1$. Furthermore, by the definition of $B$, $u_1(f(\omega),\omega) > \sup_n u_1(f_n(\omega),\omega)$ for each $\omega \in B$ while $u_1(f(\omega),\omega) = u_1(f_n(\omega),\omega) $ for each $\omega \in B^C$. Since $\Prob_1(B) > 0$ and $x\mapsto u_1(x,\omega)$ is increasing, by Monotone Convergence Theorem we have:
\begin{equation*}
    \lim_n E_{\Prob_1}[u_1(f_n,\cdot)] =  E_{\Prob_1}[\sup_n u_1(f_n,\cdot)]  < E_{\Prob_1}[u_1(f,\cdot)]
\end{equation*}
By continuity and strict monotonicity of $u_0$, there exists $a\in\R$ such that
\begin{equation*}
    \sup_n E_{\Prob_1}[u_1(f_n,\cdot)] < u_0(a) < E_{\Prob_1}[u_1(f,\cdot)]
\end{equation*}
that is $a\succ_{0,1} f_n$ $\forall n$ while $a \prec_{0,1} f$. This contradicts Axiom (C.0), hence we conclude that $\Prob_1(B)$ equals zero.

\medskip

\paragraph{Proof of ($\Leftarrow$)} Viceversa, we assume that the preference $\succcurlyeq_{0,1}$ is given by:
\begin{equation*}
    a \succcurlyeq_{0,1} f \iff u_0(a) \geq V_1(f)
\end{equation*}
for $a\in \R$, $f\in \cl^\infty(\Fcal_{t_1})$, with $V_1(f) = E_{\Prob_1}[u_1(f,\cdot)]$,
where $u_1$ and $\Prob_1$ are given as in Proposition \ref{main:0}. We want to show that that $\succcurlyeq_{0,1}$ satisfies Axioms (T.0), (M.0), (ST.0) and (C.0).

Let $a\in \R$ and $f\in \cl^\infty(\Fcal_{t_1})$. Clearly either $u_0(a) \leq V_1(f)$ or $u_0(a) \geq V_1(f)$ hence $\succcurlyeq_{0,1}$ is complete. Consider $a,b\in\R$ and $f\in \cl^\infty(\Fcal_{t_1})$ satisfying $a\preccurlyeq_{0,1} f$ and $b \succcurlyeq_{0,1} f$. This means that $u_0(a) \leq V_1(f) \leq u_0(b)$. From the fact that $u_0$ is strictly increasing it follows that $b\geq a$, that is that $\succcurlyeq_{0,1}$ is transitive. Clearly $0\sim_{0,1} 0$ since $u_0(0) = 0 = E_{\Prob_1}[u_1(0,\cdot)]$. Finally, let $f\in \cl^\infty(\Fcal_{t_1})$. By assumption the range of $V_1$ is contained in the range of $u_0$ so that there exists $b\in\R$ such that $u_0(b) \geq V_1(f)$ and (equivalently) $b\succcurlyeq_{0,1} f$. For the same reason there exists $a\in\R$ such that $u_0(a)\leq V_1(f)$, that is $a\preccurlyeq_{0,1} f$. This means that $\succcurlyeq_{0,1}$ is non-degenerate concluding the proof that Axiom (T.0) holds.

Let $a,b,c \in \R$ with $a<b$, $f \in \cl^\infty(\Fcal_{t_1})$ and $A\in \Fcal_{t_1}$ being non-null. Suppose that $c\sim_{0,1} a\mathbf{1}_A + f \mathbf{1}_{A^C}$, that is $ u_0(c) = E_{\Prob_1}[u_1(a\mathbf{1}_A + f \mathbf{1}_{A^C},\cdot)] = E_{\Prob_1}[u_1(a,\cdot)\mathbf{1}_A+ u_1(f,\cdot)\mathbf{1}_{A^C})]$.

\noindent Now, since $u_1(\cdot,\omega)$ is strictly increasing for each $\omega$, then $E_{\Prob_1}[u_1(a,\cdot)\mathbf{1}_A] < E_{\Prob_1}[u_1(b,\cdot)\mathbf{1}_A]$.
Then $ u_0(c) < E_{\Prob_1}[u_1(b,\cdot)\mathbf{1}_A+ u_1(f,\cdot)\mathbf{1}_{A^C})] = E_{\Prob_1}[u_1(b\mathbf{1}_A + f \mathbf{1}_{A^C},\cdot)]$ which means
$c\prec_{0,1} b\mathbf{1}_A + f \mathbf{1}_{A^C}$. The same argument can be used for $c\sim_{0,1} b\mathbf{1}_A + f\mathbf{1}_{A^C}$ leading to $c\succ_{0,1} a\mathbf{1}_A + f\mathbf{1}_{A^C}$. Thus, $\succcurlyeq_{0,1}$ satisfies (M.0).

(ST.0) follows from the simple fact that $E_{\Prob_1}[u_1(f\mathbf{1}_A+h\mathbf{1}_{A^c},\cdot)]\leq E_{\Prob_1}[u_1(g\mathbf{1}_A+h\mathbf{1}_{A^c},\cdot)]$ if and only if 
$E_{\Prob_1}[u_1(f\mathbf{1}_A+k\mathbf{1}_{A^c},\cdot)]\leq E_{\Prob_1}[u_1(g\mathbf{1}_A+k\mathbf{1}_{A^c},\cdot)]$ whatever the choice of $A\in \Fcal_{t_1}\setminus \cn(\Fcal_{t_1})$ and $f,g,h,k \in \cl(\Fcal_{t_1})$.

Finally, let $(f_n)_{n\in\mathbb{N}}\subseteq \cl^\infty(\Fcal_{t_1})$ be uniformly bounded and converging pointwise to $f$ for each $\omega\in\Omega$. Let $K := \sup_n \norm{f_n}_\infty \in \R^+$. Since the integral representation is pointwise continuous (on uniformly bounded sequences) we have:
\begin{equation}\label{eq:integral limit}
    E_{\Prob_1}[u_1(f_n,\cdot)] \to E_{\Prob_1}[u_1(f,\cdot)]
\end{equation}
Now let $a\in \R$ such that $a \prec_{0,1} f$ and call $\varepsilon := E_{\Prob_1}[u_1(f,\cdot)]-u_0(a) >0$. Then by (\ref{eq:integral limit}) there exists $N\in\mathbb{N}$ such that $|E_{\Prob_1}[u_1(f,\cdot)] -E_{\Prob_1}[u_1(f_n,\cdot)]| < \varepsilon$ $\forall n >N$. The triangular inequality implies that $u_0(a) < E_{\Prob_1}[u_1(f_n,\cdot)]$ $\forall n > N$, that is $a \prec_{0,1} f_n$. The same argument applies to $a \succ_{0,1} f$. Hence, (C.0) holds, concluding the proof.

%



\section{Inductive proof of Theorem \ref{main:theorem}}\label{inductive:proof}


This section is entirely devoted to the proof of the main Theorem of this paper.

\paragraph{On the  direct implication ($\Rightarrow$).}  We shall proceed by
induction. In fact if $N=1$ Theorem \ref{main:theorem} reduces to
Proposition \ref{main:0}, which is proved in the previous Section
\ref{unconditioned:updating}.

\begin{assumption}\label{induction:assumption}[Induction] We assume that the statement is true up to $i$. In particular it means
that we can guarantee the existence of a probability $\Prob_{i}$
on $\Fcal_{t_{i}}$ and state-dependent utilities 
$\{u_k\}_{k=1}^i$, where $u_k(x,\cdot)$ is
$\Fcal_{t_k}$-measurable, integrable, strictly increasing in $x$, $\star$-continuous,
$u_k(0,\cdot)=0$ and 
\begin{eqnarray*}g \succeq_{k-1,k} f & \iff & u_{k-1}(g) \ge
E_{\mathbb{P}_{i}}[u_k(f)| \mathcal{F}_{t_{k-1}}] \quad
\mathbb{P}_{i}\text{-a.s.}
\\ g \preceq_{k-1,k} f & \iff & u_{k-1}(g) \le
E_{\mathbb{P}_{i}}[u_k(f)| \mathcal{F}_{t_{k-1}}] \quad
\mathbb{P}_{i}\text{-a.s.}
\end{eqnarray*}
for any $k=1,\dots,i$, $f\in
\cl^{\infty}(\Omega,\Fcal_{t_{k-1}})$, $g\in
\cl^{\infty}(\Omega,\Fcal_{t_{k}})$.
\end{assumption}

\noindent Under this assumption we shall now prove that the
representation can be forwardly updated to time $t_{i+1}$.

\begin{remark} We point out that $\mathcal{N}(\mathcal{F}_{t_{i}})=\{A\mid
\exists B\in \mathcal{F}_{t_{i}},\; \Prob_i(B)=0\text{ and }
A\subseteq B\}$, where $\mathcal{N}(\mathcal{F}_{t_{i}})$ are the
null sets induced by the relations $\preceq_{i-1,i}$,
$\succeq_{i-1,i}$ as in \eqref{null:events}.
\end{remark}

Although a conditional preference is not total, the following
lemma, which is inspired by Lemma 3.2 in \cite{DJ14}, shows that
local completeness allows to derive for every two acts an
$\mathcal{F}_{t_i}$-measurable partition on which a comparison can
be achieved.

\begin{lemma}
\label{treeventi} Consider any $g \in
\cl^{\infty}(\mathcal{F}_{t_i})$, $f \in
\cl^{\infty}(\mathcal{F}_{t_{i+1}})$. If Assumption \ref{induction:assumption} holds
and $\succeq_{i,i+1}$ satisfies (T.i) then  there exists a
pairwise disjoint family of events $A, B, C \in \Fcal_{t_i}$ such
that $\Prob_i(A \cup B \cup C) = 1$ and
\begin{align*}
g \mathbf{1}_A &\sim_{i,i+1} f \mathbf{1}_A, \\
g  &\succ_{i,i+1}^B f  \\
g & \prec_{i,i+1}^C f .
\end{align*}
\end{lemma}

\begin{proof}
Fix $g \in \cl^{\infty}(\mathcal{F}_{t_i})$, $f \in
\cl^{\infty}(\mathcal{F}_{t_{i+1}})$, and define
$\mathcal{E} := \{ \tilde{A} \in \mathcal{F}_{t_i} : g
\mathbf{1}_{\tilde{A}} \sim_{i,i+1} f \mathbf{1}_{\tilde{A}} \}$,
$S:=\sup_{\tilde{A}\in \Ecal}\Prob_{i}(\tilde{A})$. We can find
$\{A_n\}_{n}\subseteq \Ecal$ such that $\Prob_{i}(A_n)\to S$: we
have $\Prob_i(\cup_n A_n)\geq \Prob_i(A_n)$ for every $n$ which
implies $\Prob_i(\cup_n A_n)=S$ (from (T.i) and Remark
\ref{regularity} we have $\cup_n A_n\in\Ecal$). We finally show
that up to null events $\cup_n A_n$ represents the largest event
on which $g$ is conditionally equivalent to $f$: let $\tilde{A}\in
\Ecal$ and $B=\tilde{A}\setminus (\cup_n A_n)$. Then $B\cup
(\cup_n A_n)\in\Ecal$ and $\Prob_i(B\cup (\cup_n
A_n))=\Prob_i(B)+S$. Necessarily $\Prob_i(B)=0$.
\\ We therefore set $A:=\cup_n A_n$ and
consider $\Ucal := \{ \tilde{B} \in \mathcal{F}_{t_i},\;
\tilde{B}\subseteq A^c : g \mathbf{1}_{\tilde{B}} \succeq_{i,i+1}
f \mathbf{1}_{\tilde{B}}\}$. Notice that from the construction of
$A$ if we find $\tilde{B}\in\Ucal$ such that $g
\mathbf{1}_{\tilde{B}} \sim_{i,i+1} f \mathbf{1}_{\tilde{B}}$ then
$\Prob_i(\tilde{B})=0$. Following the same argument as in the
previous step we construct a maximal $B\in \Ucal$ such that
$\Prob_i(B)\geq \Prob_i(\tilde{B})$ for all $\tilde{B}\in\Ucal$:
indeed it is not possible the finding of $B'\subset B$ with $\Prob_i(B')>0$ such that
$g\mathbf{1}_{B'} \preceq_{i,i+1} f \mathbf{1}_{B'}$ and therefore
$g  \succ_{i,i+1}^B f$.
\\ Finally we can consider $\Dcal := \{ \tilde{C} \in \mathcal{F}_{t_i},\;
\tilde{C}\subseteq (A\cup B)^c : g \mathbf{1}_{\tilde{C}}
\preceq_{i,i+1} f \mathbf{1}_{\tilde{C}}\}$ and following the same
reasoning we can find $C\in\Dcal$ such that $\Prob_i(C)\geq
\Prob_i(\tilde{C})$ for all $\tilde{C}\in\Dcal$ and $g
\prec_{i,i+1}^C f$.
\\ By construction $\Prob_i(A\cup B \cup C)=1$ and the probability
of the intersections is always $0$.
\end{proof}

\noindent Consider for any $g\in
\cl^{\infty}(\mathcal{F}_{t_{i}})$ the upper and lower
level sets $ \Ccal^u_g  =  \{f\in
\cl^{\infty}(\mathcal{F}_{t_{i+1}})\mid g\preceq_{i,i+1} f
\}$ and  $\Ccal^l_g  =  \{f\in
\cl^{\infty}(\mathcal{F}_{t_{i+1}})\mid g\succeq_{i,i+1} f
\}$ and the maps
\begin{eqnarray*} V_{i+1}^-(f) & = & \Prob_i-\sup\{u_i(g) \mid f\in \Ccal^u_g
\}=\Prob_i-\sup\{u_i(g) \mid g\preceq_{i,i+1} f \}
\\ V_{i+1}^+(f) & = & \Prob_i-\inf\{u_i(g) \mid f\in \Ccal^l_g \}=\Prob_i-\inf\{u_i(g) \mid g\succeq_{i,i+1} f \}
\end{eqnarray*}

\begin{lemma}\label{equalityI} Let Assumption \ref{induction:assumption} holds
and $\succeq_{i,i+1}$ satisfies (T.i). The maps
$V_{i+1}^+(f):\cl^{\infty}(\mathcal{F}_{t_{i+1}})\to
L^{0}(\Omega,\Fcal_{t_i},\Prob_i)$,
$V^-_{i+1}(f):\cl^{\infty}(\mathcal{F}_{t_{i+1}})\to
L^{0}(\Omega,\Fcal_{t_i},\Prob_i) $ are well defined. Moreover, as
$u_i(\omega,\cdot)$ is strictly increasing and $\star$-continuous (Assumption \ref{induction:assumption}), then
$V_{i+1}^+(f)=V^-_{i+1}(f)$ for any $f\in
\cl^{\infty}(\mathcal{F}_{t_{i+1}})$.
\end{lemma}

\begin{notation} We shall often use the notation $u_i^{-1}V_{i+1}(f)$ to
indicate the function mapping $\omega \to
u_i^{-1}(V_{i+1}(f)(\omega),\omega)$ \footnote{This function is
well defined and measurable as $u_i(\cdot,\omega)$ is strictly
increasing for any $\omega\in \Omega$.}.
\end{notation}

\begin{proof} Let $g_1,g_2\in \cl^{\infty}(\mathcal{F}_{t_{i}})$ such that $\Prob_i(g_1=g_2)=1$. We have from Remark \ref{regularity} that
$\Ccal^u_{g_1}=\Ccal^u_{g_2}$ and $\Ccal^l_{g_1}=\Ccal^l_{g_2}$
and therefore $V_{i+1}^+ , V^-_{i+1}$ are well defined.
\\ From now on we fix $f \in \cl^{\infty}(\mathcal{F}_{t_{i+1}})$: for any $g_1,g_2\in \cl^{\infty}(\mathcal{F}_{t_{i}})$ such that $g_1 \preceq_{i,i+1} f$ and
$g_2 \succeq_{i,i+1} f$ then the set $\{g_1>g_2\}\in
\Ncal(\Fcal_{t_i})$. From the monotonicity of $u_i$ we have
$\Prob_i(u_i(g_1)\leq u_i(g_2))=1$ and therefore $V_{i+1}^-(f)\leq
V^+_{i+1}(f)$, $\Prob_i$ almost surely. 
\\ To prove that $V_{i+1}^-(f) = V_{i+1}^+(f)$ we need to find 
$g_-, g_+ \in \cl^{\infty}(\Fcal_{t_{i}})$
such that $u_i(g_{\pm}) = V^{\pm}_{i+1}(f)$ $\Prob_i-a.s.$. 
We prove the existence of $g_+$, then the same argument works also for $g_-$. 
Take a sequence $(g_n)_{n \in \N} \subseteq \cl \left(\Fcal_{t_i} \right)$ 
satisfying $g_n \succeq_{i,i+1} f$, 
$g_{n+1}(\omega) \leq g_n(\omega)$ $\forall n \in \N, 
\omega\in\Omega$ and $u_i(g_n(\omega),\omega)\downarrow(V^+_{i+1}(f))(\omega)$ 
for each $\omega\in A$ for some event $A \in \Fcal_{t_i}$ with $\Prob_i(A) = 1$. 
The existence of such sequence is guaranteed by the definition of 
$\Prob_i-\inf$ and the fact that the set 
$\left\{g\in \cl^{\infty}(\Fcal_{t_i}) : g \succeq_{i,i+1} f \right\}$ 
is downward directed \footnote{A set $\Acal$ is downward directed if for any $f,g\in\Acal$ the minimum $f\wedge g\in \Acal$. The existence of a minimizing sequence is proved in Appendix A.5 of \cite{FS04}}. 
Since the prefererence relation $\preceq_{i,i+1}$ is non degenerate, 
there exists an act $h \in \cl^{\infty}(\Fcal_{t_i})$ 
such that $h\preceq_{i,i+1} f$ implying that the event 
$\{ \omega\in\Omega : h(\omega)> g_n(\omega) \}\in \Fcal_{t_i}$ 
is null for each $n \in \N$. This means that the sequence $(g_n)_n$ 
is decreasing and has a $\Prob_i-a.s.$ finite lower bound, 
hence there exists an event $B\in \Fcal_{t_i}$ with $\Prob_i(B)=1$ 
and an act $g_+ \in \cl^{\infty}(\Fcal_{t_i})$ 
such that $g_n(\omega)\downarrow g_+(\omega)\in \R$ 
for all $\omega \in B$. The $\star$-contintuity of $u_i$ ensures 
that $g_n(\omega)$ and $g_+(\omega)$ belong to the points of 
(right) continuity of $u_i(\cdot,\omega)$ for each $\omega\in C$
for some $C \in \Fcal_{t_i}$ with $\Prob_i(C)=1$. This leads to:
\begin{equation*}
	(V^+_{i+1}(f))(\omega)=\lim_n u_i(g_n(\omega),\omega)=u_i(g_+(\omega),\omega)
\end{equation*} 
for each $\omega \in A\cap B \cap C$ and $\Prob_i(A \cap B \cap C)=1$.
\\ Consider now $\bar A\in\Fcal_{t_i}$ defined by $\bar A:= \{g_-< g_+\}$.
For $\lambda\in (0,1)$ define
the convex combination $g_{\lambda}:=\lambda g_+ +
(1-\lambda)g_-$. Indeed $\bar A = \{ g_{\lambda}< g_+\}=\{
u_i(g_{\lambda})< u_i(g_+)\}=$ $\{g_{\lambda}>
g_-\}=\{u_i(g_{\lambda})> u_i(g_-)\}$.
\\ Observe that if $\Prob_i(\bar A)=0$ we have the thesis. Otherwise we claim that for any $B \subseteq \bar A$, $B\in\Fcal_{t_i}$,
$\Prob_i(B)>0$ neither $g_{\lambda}\mathbf{1}_B \preceq_{i,i+1}
f\mathbf{1}_B$ nor $g_{\lambda}\mathbf{1}_B \succeq_{i,i+1}
f\mathbf{1}_B$ occur. This claim indeed contradicts local
completeness in (T.i).
\\To show the claim we consider first the
case $g_{\lambda}\mathbf{1}_B \preceq_{i,i+1} f\mathbf{1}_B$ for
 some $B\subseteq \bar A$, $B\in\Fcal_{t_i}$ and  $\Prob_i(B)>0$, since the other
follows in a similar way. As a consequence of Remark
\ref{regularity} we have $g_{\lambda}\mathbf{1}_B +g_-
\mathbf{1}_{B^c} \preceq_{i,i+1} f$. From the construction
$B\setminus\{u_i(g_{\lambda}\mathbf{1}_B +g_-
\mathbf{1}_{B^c})>V^-(f)\}$ is null, but from the definition
$\Prob_i-\sup$ we necessarily have $\{u_i(g_{\lambda}\mathbf{1}_B
+g_- \mathbf{1}_{B^c})>V^-(f)\}\in \Ncal(\Fcal_{t_i})$. Therefore
$g_{\lambda}\mathbf{1}_B \preceq_{i,i+1} f\mathbf{1}_B$ cannot
occur for any for $B\subseteq \bar A$, $B\in\Fcal_{t_i}$ and
$\Prob_i(B)>0$. Similarly we can obtain that
$g_{\lambda}\mathbf{1}_B \succeq_{i,i+1} f\mathbf{1}_B$ cannot
occur for any for $B\subseteq \bar A$, $B\in\Fcal_{t_i}$ and
$\Prob_i(B)>0$, concluding the proof of the claim.
 \end{proof}

\begin{notation}
From now on we shall denote $V_{i+1} := V_{i+1}^+ = V^-_{i+1}$.
\end{notation}

\begin{proposition}Let Assumption \ref{induction:assumption} holds
and $\succeq_{i,i+1}$ satisfies (T.i). Then for any $f\in
\cl^{\infty}(\mathcal{F}_{t_{i+1}})$ there exists a unique
Conditional Certainty Equivalent given by
$C_{i,i+1}(f)=u_i^{-1}V_{i+1}(f) \in
L^{\infty}(\Omega,\Fcal_{t_i},\Prob_i) $. Moreover $V_{i+1}$
represents the transition order i.e.
\begin{eqnarray}
g\preceq_{i,i+1} f & \Leftrightarrow & u_i(g)\leq V_{i+1}(f) \quad
\Prob_i\text{-a.s.} \label{repr:stepi}
\\ g\succeq_{i,i+1} f & \Leftrightarrow & u_i(g)\geq V_{i+1}(f)
\quad \Prob_i\text{-a.s.} \label{repr:stepi:opposite}
\end{eqnarray}
and necessarily $V_{i+1}(f)\in L^{1}(\Omega,\Fcal_{t_i},\Prob_i)$.
\end{proposition}

\begin{proof} In this proof we denote (with a slight abuse of notation) by $V_{i+1}(f)$ any of its $\Fcal_{t_i}$-measurable version.
Existence and uniqueness follow from the previous Lemma
\ref{equalityI}. We only need to show that
$C_{i,i+1}(f)=u_i^{-1}V_{i+1}(f) \in
L^{\infty}(\Omega,\Fcal_{t_i},\Prob_i) $. For any couple $g_1,g_2
\in \cl^{\infty}(\Fcal_{t_i})$ such that
$g_1\preceq_{i,i+1} f$ and $g_2\succeq_{i,i+1} f$ we can observe
that $u_i(g_1)\leq V_{i+1}(f)\leq u_i(g_2)$, $\Prob_i$ almost
surely, which automatically implies $V_{i+1}(f)\in
L^{1}(\Omega,\Fcal_{t_i},\Prob_i)$ (we are assuming $u_i(\cdot,x)$
is integrable for any $x$). At the same time from $u_i$ strictly
increasing in $x$ we can deduce $g_1\leq C_{i,i+1}(f)\leq g_2$,
$\Prob_i$ almost surely.

To show the representation property \eqref{repr:stepi} and
\eqref{repr:stepi:opposite}, we consider the case
$g\preceq_{i,i+1} f$ as $g\succeq_{i,i+1} f$ follows in a similar
fashion. Obviously $g\preceq_{i,i+1} f$ implies
$\Prob_i(u_i(g)\leq V_{i+1}(f))=1$ (from the definition of
$V_{i+1}^+=V_{i+1}$).
\\ For the reverse implication notice that on the set $A=\{u_i(g)=V_{i+1}(f)\}$ we necessarily have
$g\mathbf{1}_A \sim_{i,i+1} f\mathbf{1}_A$. If instead we consider
$A=\{u_i(g)<V_{i+1}(f)\}$ then either $\Prob_i(A)=0$ or necessary
$g\mathbf{1}_A \preceq_{i,i+1} f\mathbf{1}_A$ and $g\mathbf{1}_B
\nsim_{i,i+1} f\mathbf{1}_B$ for any $B\subset A$,
$B\in\Fcal_{t_i}$ as $V_{i+1}(f)$ is by definition
$\Prob_i-\inf\{u_i(g) \mid g\succeq_{i,i+1} f \}$ (This can be
easily verified applying (T.i)).
\end{proof}

\begin{corollary}\label{towk} Let Assumption \ref{induction:assumption} holds
and $\succeq_{i,i+1}$ satisfies (T.i). For any $f\in
\cl^{\infty}(\mathcal{F}_{t_{i+1}})$ and $A\in
\Fcal_{t_i}$ we have $V_{i+1}(f\mathbf{1}_A) =
V_{i+1}(f)\mathbf{1}_A$, $\Prob_i$ almost surely.
\end{corollary}

\begin{proof} From the previous construction we have $ u_i^{-1}V_{i+1}(f\mathbf{1}_A)\sim_{i,i+1} f\mathbf{1}_A$. Moreover from (T.i) we also have that $u_i^{-1}V_{i+1}(f) \sim_{i,i+1} f$ implies $u_i^{-1}V_{i+1}(f)\mathbf{1}_A \sim_{i,i+1} f\mathbf{1}_A$. Therefore from transitivity we deduce $u_i^{-1}V_{i+1}(f)\mathbf{1}_A = u_i^{-1}V_{i+1}(f \mathbf{1}_A)$, $\Prob_i$ almost surely and hence the thesis.
\end{proof}

\begin{remark}\label{strict} Let Assumption \ref{induction:assumption} holds
and $\succeq_{i,i+1}$ satisfies (T.i). For any $f\in
\cl^{\infty}(\mathcal{F}_{t_{i+1}})$, $g\in
\cl^{\infty}(\mathcal{F}_{t_{i}})$ and $A\in \Fcal_{t_i}$
we have $g \prec_{i,i+1}^A f$ (resp. $g \succ_{i,i+1}^A f$)
implies $\{u_i(g)\geq V_{i+1}(f)\}\cap A\in
\mathcal{N}(\mathcal{F}_{t_{i}})$ (resp. $\{u_i(g)\leq
V_{i+1}(f)\}\cap A\in \mathcal{N}(\mathcal{F}_{t_{i}})$)
\end{remark}

\paragraph{Last step of the proof for ($\Rightarrow$):}  Let Assumption \ref{induction:assumption} holds
and $\succeq_{i,i+1}$ satisfies all the Axioms (T.i), (M.i),
(ST.i) and (C.i). In order to conclude the proof we show that
there exist a probability $\mathbb{P}_{i+1}$ on
$(\Omega,\Fcal_{t_{i+1}})$ which agrees with $\Prob_i$ on
$\Fcal_{t_i}$ and a state-dependent utility $u_{i+1}(\omega, \cdot) :
\mathbb{R} \to \mathbb{R}$ strictly increasing $\forall \omega \in
\Omega$, such that
\begin{equation}\label{repr.Stepi}
V_{i+1}(f)= E_{\Prob_{i+1}}[u_{i+1}(\cdot, f)\mid \Fcal_{t_i}]
\quad \Prob_{i}\text{-a.s.}.
\end{equation}

We define an intertemporal preference relation between time $0$
and $t_{i+1}$ as $a\preceq_{0,i+1} f$ (resp. $a\succeq_{0,i+1} f$)
if and only if $u_0(a)\leq E_{\Prob_i}[V_{i+1}(f)]$ (resp.
$u_0(a)\geq E_{\Prob_i}[V_{i+1}(f)]$) for any $a\in\R$ and
$f\in\cl^{\infty}(\mathcal{F}_{t_{i+1}})$.
\\ Simple inspections show that $\preceq_{0,i+1}$ satisfies (T.0), (M.0) and (ST.0).
\\We now prove the continuity (C.0) of $\preceq_{0,i+1}$: consider
any uniformly bounded sequence $\{f_n\}\subseteq
\cl^{\infty}(\mathcal{F}_{t_{i+1}})$, such that
$f_n(\omega)\rightarrow f(\omega)$ for any $\omega\in\Omega$.
Consider $a\prec_{0,i+1}f$ (the case  $a\succ_{0,i+1}f$ follows in
a similar way) so that we necessarily have $u_0(a)<
E_{\Prob_i}[V_{i+1}(f)]$. It is possible to find $g\in
\cl^{\infty}(\mathcal{F}_{t_{i}})$ such that $u_i(g) <
V_{i+1}(f)$ and $u_0(a)< E_{\Prob_i}[u_i(g)]$ \footnote{To show
the existence of such $g$ we need to consider for any $\varepsilon
>0$, $C_{i,i+1}(f)-\varepsilon$ so that
$u_i(C_{i,i+1}(f)-\varepsilon)< u_i(C_{i,i+1}(f))=V_{i+1}(f)$;
observing that $u_i(C_{i,i+1}(f)-\varepsilon)$ increases
monotonically to $u_i(C_{i,i+1}(f))$ (for any $\omega\in \Omega$)
we can find an $\bar\varepsilon$ such that $u_0(a)<
E_{\Prob_i}[u_i(C_{i,i+1}(f)-\bar\varepsilon)]<
E_{\Prob_i}[u_i(C_{i,i+1}(f))]=E_{\Prob_i}[V_{i+1}(f)]$.}. \\Since
$g\prec_{i,i+1}f$ we apply (C.i) and find a sequence of indexes
$\{n_k\}_{k=1}^{\infty}$ and a partition
$\{A_k\}_{k=1}^{\infty}\subset \Fcal_{t_i}$ such that for any $k$
we have $g\mathbf{1}_{A_k}\preceq_{i,i+1}f_{n}\mathbf{1}_{A_k}$
for all $n\geq n_k$.

For $B_N=\cup_{i=1}^N A_i$ and $d=\sup_n\|f_n\|_{\infty}$ consider
the CCE  $C_{i,i+1}(-d)$. The sequence
$\{u_i(g\mathbf{1}_{B_N}+C_{i,i+1}(-d)\mathbf{1}_{B_N^c}\}_{N\in
\N}$ is dominated by the integrable function
$|u_i(g)|+|u_i(C_{i,i+1}(-d))|$ and pointwise converges to
$u_i(g)$. From Dominated Convergence Theorem we can find $\bar N$
such that
$$E_{\Prob_i}[u_i(g\mathbf{1}_{B_{\bar N}}+C_{i,i+1}(-d)\mathbf{1}_{B_{\bar N}^c})]>u_0(a),$$
so that from (T.i) we can deduce $u_i(g\mathbf{1}_{B_{\bar
N}}+C_{i,i+1}(-d)\mathbf{1}_{B_{\bar N}^c})\leq
V_{i+1}(f_n\mathbf{1}_{B_{\Bar N}}-d\mathbf{1}_{B_{\Bar N}^c})$
for $n>\bar N$ and
$$E_{\Prob_i}[V_{i+1}(f_n)]\geq E_{\Prob_i}[V_{i+1}(f_n\mathbf{1}_{B_{\bar N}}-d\mathbf{1}_{B_{\bar N}^c})]>u_0(a),\quad \forall \, n> \bar N,$$
which shows (C.0) of $\preceq_{0,i+1}$.

\medskip

Given that $\preceq_{0,i+1}$ satisfies (T.0), (M.0), (ST.0) and (C.0)  premise we can apply Proposition \ref{main:0}
and find a probability $\widetilde{\Prob}$ on $\Fcal_{t_{i+1}}$
and a state-dependent utility $\widetilde{u}$ such that
$E_{\Prob_i}[V_{i+1}(f)]=E_{\widetilde{\Prob}}[\widetilde{u}(f)]$
for any $f\in \cl^{\infty}(\mathcal{F}_{t_{i+1}})$.
\\Notice from (T.i) point 3 that $\widetilde{\Prob}$ is equivalent to $\Prob_i$ on
$\Fcal_{t_i}$. For $\widetilde{\Prob}_{|\Fcal_{t_i}}$ being the
restriction of $\widetilde{\Prob}$ on $\Fcal_{t_i}$ define
$Z=\frac{d\Prob_i}{d\widetilde{\Prob}_{|\Fcal_{t_i}}}$, which is
an $\Fcal_{t_i}$-measurable random variable. For any $A\in
\Fcal_{t_{i+1}}$ set
$\Prob_{i+1}(A):=E_{\widetilde{\Prob}}[Z\mathbf{1}_A]$,
$u_{i+1}(\omega,x)=\frac{d\widetilde{\Prob}}{d\Prob_{i+1}}\widetilde{u}(\omega,x)$
and notice that $\Prob_{i+1}(A)=\Prob_{i}(A)$ for any
$A\in\Fcal_{t_i}$. We have

$$ E_{\Prob_i}[V_{i+1}(f)] =E_{\widetilde{\Prob}}[\widetilde{u}(f)]=E_{\Prob_{i+1}}[u_{i+1}(f)]= E_{\Prob_{i}}[E_{\Prob_{i+1}}[u_{i+1}(f)\mid \Fcal_{t_i}]],$$
so that we can obtain for every $A\in\Fcal_{t_i}$ that
\begin{eqnarray*} E_{\Prob_i}[V_{i+1}(f)\mathbf{1}_A]=
E_{\Prob_i}[V_{i+1}(f\mathbf{1}_A)] &=&
E_{\Prob_{i}}[E_{\Prob_{i+1}}[u_{i+1}(f\mathbf{1}_A)\mid
\Fcal_{t_i}]]
\\& = & E_{\Prob_{i}}[E_{\Prob_{i+1}}[u_{i+1}(f)\mid
\Fcal_{t_i}]\mathbf{1}_A],
\end{eqnarray*}
which implies the representation \eqref{repr.Stepi}.

We finally show the $\star$-continuity of $u_{i+1}$. As for the unconditional case, 
it is enough to show that for each $f\in \mathscr{L}(\Fcal_{t_{i+1}})$ it holds that 
$\Prob_{i+1}(LD_f) = 0$ where $LD_f$ is defined in Appendix \ref{Appendix A} 
with respect to the stochastic field $u_{i+1}$. Notice that, as consequence of Lemma 
\ref{lma:measurable sets}, $\forall f\in \mathscr{L}(\Fcal_{t_{i+1}})$ then 
$LD_f \in \Fcal_{t_{i+1}}$, hence either $\Prob_{i+1}(LD_f) = 0$ or 
$\Prob_{i+1}(LD_f) > 0$. By contradiction, we assume that there exists an act 
$f^* \in \mathscr{L}^\infty(\Fcal_{t_{i+1}})$ for which $\Prob_{i+1}(LD_{f^*}) > 0$. 
In order to simplify the notation we set $B := LD_f$ and, since the 
probability $\Prob_{i+1}$ is fixed, we denote $\Prob_{i+1}-\sup(\Acal)$ 
simply with  $\sup(\Acal)$   for any family $\Acal \subseteq \Lcal^\infty(\Fcal_{t_{i+1}})$. 
Define $f := f^* \ob_B$ and $f_n := \left(f-\frac{1}{n}\right)\ob_B$ for each $n \in \N$. 
Clearly $f_n \to f$ in $\mathscr{L}^\infty(\Fcal_{t_{i+1}})$, $\norm{f_n} \leq \norm{f} + 1 < +\infty$ $\forall n$ 
and $f_n(\omega) = f(\omega)=0$ $\forall \omega \in B^C$. By definition of $B$, 
it hods that $u_{i+1}(f(\omega),\omega) > \sup_n u_{i+1}(f_n(\omega),\omega)$ 
for $\Prob_{i+1}$-a.e. $\omega\in B$ and, so, we have:
\begin{eqnarray}\label{eq:contr}
	& \Prob_i\left( \Ea_{\Prob_{i+1}}[u_{i+1}(f,\cdot) | \Fcal_{t_i}] > \sup_n \Ea_{\Prob_{i+1}}[u_{i+1}(f_n,\cdot)|\Fcal_{t_i}] \right) > 0 \\
	& \Ea_{\Prob_{i+1}}[u_{i+1}(f,\cdot)|\Fcal_{t_i}] \geq \sup_n \Ea_{\Prob_{i+1}}[u_{i+1}(f_n,\cdot)|\Fcal_{t_i}] \quad \Prob_{i+1}-a.s. \nonumber
\end{eqnarray}
Define now $g_n := C_{i,i+1}(f_n)$ and $g := C_{i,i+1}(f)$. Observe that $\{g_n\}_n$ 
is an increasing sequence as $\{f_n\}_n$ increases and $u_i(\cdot,\omega)$ is strictly 
increasing for each $\omega$ by Assumption \ref{induction:assumption} and has $g$ 
as upper bound. If $g_n(\omega) \to g(\omega)$ for $\Prob_i$-a.e. $\omega \in \Omega$ 
then, by the $\star$-continuity of $u_i$, it would happen that 
$u_i(g_n(\omega),\omega) \to u_i(g(\omega),\omega)$ for $\Prob_i$-a.e. 
$\omega \in \Omega$ in contradiction with \eqref{eq:contr}. Hence, 
there exists $A \in \Fcal_{t_i}$ with $\Prob_i(A)>0$ such that 
$\sup_n g_n(\omega) < g(\omega)$ for each $\omega \in A$. 
Take now $\lambda \in (0,1)$ and consider 
$g_\lambda := \lambda g + (1-\lambda ) \sup_n g_n$. It holds that:
\begin{eqnarray*}
	\sup_n g_n(\omega) \leq g_\lambda \leq g(\omega) \quad & \text{for } \Prob_i-a.e.\; \omega \in \Omega \\
\sup_n g_n(\omega) < g_\lambda < g(\omega) \quad & \text{for each } \omega \in A
\end{eqnarray*}
Therefore, it follows that $g_\lambda \succeq_{i,i+1} f_n$ and 
$g_\lambda \succ^A_{i,i+1} f_n$ $\forall n$, while  $g_\lambda \preceq_{i,i+1} f$ 
and $g_\lambda \prec^A_{i,i+1} f$ which is in contradiction with axiom (C.i). 

\medskip

\paragraph{On the reverse implication ($\Leftarrow$).} We now assume that there exist a probability $\Prob$ on
$\Fcal_{t_N}$ and a Stochastic Dynamic Utility $u(t,x,\omega)$ in
the form of \eqref{SDU} with properties (a) (b) (c) and (d).
%
Then, for any $i=1,\dots,N-1$, it is easy to show that the
intertemporal preferences $\succeq_{i,i+1}, \preceq_{i,i+1}$
satisfy Axioms (T.i), (M.i), (ST.i), from the properties of the
conditional expectation and the monotonicity of the Stochastic Dynamic Utility. \\The
only critical point is showing property (C.i). To this aim let
$\{f_n\}\subseteq \cl^{\infty}(\mathcal{F}_{t_{i+1}})$ be
a uniformly bounded sequence, such that $f_n(\omega)\rightarrow
f(\omega)$ for any $\omega\in\Omega$. Choose any $g\prec_{i,i+1}f$
then necessarily $\Prob(u(t_i, g) \ge E_{\mathbb{P}}[u(t_{i+1},f)|
\mathcal{F}_{t_i}])=0$.
\\As $\sup_n\|f_n\|_{\infty}<d$ for some $d>0$ we build the increasing
sequence $l_n:=\inf_{k\geq n}f_k\in
\cl^{\infty}(\mathcal{F}_{t_{i+1}})$ and notice $l_n\leq
f_n$ and $l_n(\omega)\rightarrow f(\omega)$ for any
$\omega\in\Omega$. Moreover $\|l_n\|_{\infty}<d$ for all $n\in \N$
and consequently $|u(t_{i+1},l_n)|\leq |u(t_{i+1},d)|$ which is
integrable. We can apply the Dominated Convergence Theorem for
conditional expectation and obtain $E_{\mathbb{P}}[u(t_{i+1},l_n)|
\mathcal{F}_{t_i}](\omega)\to E_{\mathbb{P}}[u(t_{i+1},f)|
\mathcal{F}_{t_i}](\omega)$ for any $\omega\in\Omega$ (by choosing
an opportune version of the conditional expectation). Consider the
sequence of sets $\{B_n\}_{n\in\N}\subset \Fcal_{t_i}$ defined by
$$B_k:=\{u(t_i, g) < E_{\mathbb{P}}[u(t_{i+1},l_k)|
\mathcal{F}_{t_i}]\}.$$ Indeed $\cup_k B_k=\Omega$ from the
pointwise convergence and we deduce that the pairwise disjoint
family $A_1:=B_1,\dots, A_{k}:=B_{k}\setminus (\cup_{i=1}^{k-1}
A_i)$ satisfies again $\cup_k A_k=\Omega$ and forms therefore a
partition of $\Omega$. We conclude by observing that for any $n\geq
k$ we have $f_n\geq l_k$ and therefore $u(t_i, g)(\omega) <
E_{\mathbb{P}}[u(t_{i+1},f_n)| \mathcal{F}_{t_i}](\omega)$ for any
$\omega\in A_k$. Finally for every  $n\geq k$ we deduce
$g\mathbf{1}_{A_k}\preceq_{i,i+1} f_n \mathbf{1}_{A_k}$, as the follwoing identies $u(t_i, g)\mathbf{1}_{B_k}= u(t_i, g\mathbf{1}_{B_k})$ and 
$E_{\mathbb{P}}[u(t_{i+1},f_n)| \mathcal{F}_{t_i}]\mathbf{1}_{B_k}=E_{\mathbb{P}}[u(t_{i+1},f_n\mathbf{1}_{B_k})| \mathcal{F}_{t_i}]$ hold $\Prob$-a.s.. The
argument repeats in the same way when $g\succ_{i,i+1}f$.

\paragraph{On the uniqueness.} To conclude the proof we need to show the relative
uniqueness. Consider the new couple $(\mathbb{P}^*,u^*)$ such that
$\mathbb{P}$ is equivalent to $\mathbb{P}^*$ on $\Fcal_{t_N}$ and
for any $i=1,\dots,N$ we have $\Prob(u^*(t_i,\cdot,\cdot) =
\delta_i u_{i})=1$, where $\delta_i$ is the Radon-Nikodym
derivative of $\mathbb{P}_{|\Fcal_{t_i}}$ with respect to
$\mathbb{P}^*_{|\Fcal_{t_i}}$. 
We show for any arbitrary $i=1,\dots,N-1$,
$g\in \cl^{\infty}(\mathcal{F}_{t_i}),f\in
\cl^{\infty}(\mathcal{F}_{t_{i+1}})$ the first of the
following equivalences
\begin{eqnarray*} u^*(t_i, g) \ge
E_{\mathbb{P}^*}[u^*(t_{i+1},f)| \mathcal{F}_{t_i}] \quad
\mathbb{P}^* \text{-a.s.}& \iff & u(t_i, g) \ge
E_{\mathbb{P}}[u(t_{i+1},f)| \mathcal{F}_{t_i}] \quad
\mathbb{P}\text{-a.s.}
\\ u^*(t_i, g) \le
E_{\mathbb{P}^*}[u^*(t_{i+1},f)| \mathcal{F}_{t_i}] \quad
\mathbb{P}^* \text{-a.s.} & \iff & u(t_i, g) \le
E_{\mathbb{P}}[u(t_{i+1},f)| \mathcal{F}_{t_i}] \quad
\mathbb{P}\text{-a.s.},
\end{eqnarray*}
as the second one is similar. To this aim we recall the
martingality property
$$\delta_i=E_{\Prob^*}\left[\frac{d\Prob}{d\Prob^*}\mid \Fcal_{t_i}\right]=E_{\Prob^*_{|\Fcal_{t_{i+1}}}}[\delta_{i+1}\mid \Fcal_{t_i}]\quad
\Prob^*\text{-a.s.}.$$ and the conditional change of measure

\begin{equation}\label{conditional:change}\frac{E_{\Prob^*}\left[\delta_{i+1}u_{i+1}(f)\mid\Fcal_{t_i}\right]}{E_{\Prob^*}\left[\delta_{i+1}\mid\Fcal_{t_i}\right]}
=E_{\Prob}\left[u_{i+1}(f)\mid\Fcal_{t_i}\right]\quad
\Prob\text{-a.s..}
\end{equation}
Moreover the equivalence between $\Prob$
and $\Prob^*$ allows to write the following inequalities
indifferently in the $\Prob$/$\Prob^*$ almost sure sense so that
we obtain
\begin{eqnarray*} u^*(t_i, g) \ge
E_{\mathbb{P}^*}[u^*(t_{i+1},f)| \mathcal{F}_{t_i}]& \iff &
\delta_i u_i(g)\ge E_{\mathbb{P}^*}[\delta_{i+1} u_{i+1}(f)|
\mathcal{F}_{t_i}] \;
\\ & \iff & \delta_i u_i(g) \ge
E_{\mathbb{P}}[u_{i+1}(f)| \mathcal{F}_{t_i}]\cdot
E_{\Prob^*}\left[\delta_{i+1}\mid\Fcal_{t_i}\right]
\\ & \iff &  u(t_i,g) \ge
E_{\mathbb{P}}[u(t_{i+1},f)| \mathcal{F}_{t_i}].
\end{eqnarray*}

On the contrary suppose that  $(\mathbb{P}^*,u^*)$ are given in a
way such that for $i=1,\dots,N$:
$E_{\Prob^*}[|u^*(t_i,x,\cdot)|]<\infty$, for all $x\in\R$,
$u^*(t_i,\cdot,\omega)$ is strictly increasing in $x$,
$u^*(t_i,0,\omega)=0$ for all $\omega\in \Omega$ and
\begin{eqnarray*} u^*(t_{i-1}, g) \ge
E_{\mathbb{P}^*}[u^*(t_{i},f)| \mathcal{F}_{t_{i-1}}] \quad
\mathbb{P}^* \text{-a.s.}& \iff & u(t_{i-1}, g) \ge
E_{\mathbb{P}}[u(t_{i},f)| \mathcal{F}_{t_{i-1}}] \quad
\mathbb{P}\text{-a.s.}
\\ u^*(t_{i-1}, g) \le
E_{\mathbb{P}^*}[u^*(t_{i},f)| \mathcal{F}_{t_{i-1}}] \quad
\mathbb{P}^* \text{-a.s.} & \iff & u(t_{i-1}, g) \le
E_{\mathbb{P}}[u(t_{i},f)| \mathcal{F}_{t_{i-1}}] \quad
\mathbb{P}\text{-a.s.},
\end{eqnarray*}
for any arbitrary, $g\in
\cl^{\infty}(\mathcal{F}_{t_i}),f\in
\cl^{\infty}(\mathcal{F}_{t_{i+1}})$. The equivalence of
$\Prob$ and $\Prob^*$ follows immediately. Moreover it is important to observe that the preferences $\succeq_{i-1,i}$ induced by $(\Prob,u)$ and $(\Prob^*,u^*)$ are the same  and satisfy all the axioms (in virtue of the previous point of the proof), which in particular implies that the CCE always exists. Moreover for any
$\omega\in \Omega$ we imposed $u(t_i,0,\omega)=
u^*(t_i,0,\omega)=0$. For $i=1$ we already know
$\Prob(u^*(t_1,\cdot,\cdot) = \delta_1 u_{1})=1$ from Proposition
\ref{main:0}. Let
$\delta_i=E_{\Prob^*}\left[\frac{d\Prob}{d\Prob^*}\mid
\Fcal_{t_i}\right]$ as before and consider the first
$i=2,\dots,N$ such that either the set $A=\{\omega \in \Omega \mid
u^*(t_i,\cdot,\omega)
> \delta_i u_{i}(\cdot,\omega)\}$ or $A=\{\omega \in \Omega \mid
u^*(t_i,\cdot,\omega) < \delta_i u_{i}(\cdot,\omega)\}$ have
positive probability. Let $C_{i-1,i}(\mathbf{1}_A)$ be the CCE of
$\mathbf{1}_A$, which is the equal under $(\Prob,u)$ or $(\Prob^*,u^*)$. Therefore  
\begin{eqnarray}\label{prima:eq} u^*(t_{i-1}, C_{i-1,i}(\mathbf{1}_A)) =
E_{\mathbb{P}^*}[u^*(t_{i},\mathbf{1}_A)| \mathcal{F}_{t_{i-1}}]
\quad \mathbb{P}^* \text{-a.s.}& \text{and} & \\ u(t_{i-1},
C_{i-1,i}(\mathbf{1}_A)) = E_{\mathbb{P}}[u(t_{i},\mathbf{1}_A)|
\mathcal{F}_{t_{i-1}}] \quad \mathbb{P}\text{-a.s.} && \nonumber
\end{eqnarray}
By performing a conditional change of measure as in \eqref{conditional:change}, the second equation can be rewritten as
$$\delta_{i-1}u(t_{i-1},
C_{i-1,i}(\mathbf{1}_A)) = E_{\mathbb{P}^*}[\delta_{i}u(t_{i},\mathbf{1}_A)|
\mathcal{F}_{t_{i-1}}] \quad \mathbb{P}^*\text{-a.s.}.$$
Subtracting this last equation and \eqref{prima:eq}, would lead to a contradiction since
the left hand side is always equal to $0$ ($\Prob$-a.s.) whereas
the right hand side is not. Therefore $\Prob(A)$ is necessarily $0$. 
\appendix

\section{On $\star$-continuity}\label{Appendix A}

Throughout this section we fix a probability space $(\Omega,\Gcal,\Prob)$ and a random field $\phi:\R \times \Omega\to\R$ such that for each $f\in \cl^\infty(\Gcal)$ the map $\omega \mapsto \phi(f(\omega),\omega)$ is $\Gcal$-measurable and for any $\omega$, $x\mapsto \phi(x,\omega)$ is non decreasing.
For any $f\in \cl^\infty(\Gcal)$ we set $$\phi(f(\omega)^+,\omega)=\inf_{n\in\mathbb{N}}\phi(f(\omega)+1/n,\omega)  \text{ and }
\phi(f(\omega)^-,\omega)=\sup_{n\in\mathbb{N}}\phi(f(\omega)-1/n,\omega)$$
and define the following sets: 
\begin{align*}
    RD_f & = \{\omega\in \Omega : \left(\phi(f(\omega)^+,\omega)-\phi(f(\omega),\omega)\right) >0\} \\ 
    LD_f & = \{\omega\in \Omega : \left(\phi(f(\omega),\omega)-\phi(f(\omega)^-,\omega)\right) >0\} \\ 
    D_f & = \{\omega\in \Omega : \left(\phi(f(\omega)^+,\omega)-\phi(f(\omega)^-,\omega)\right) >0\}
\end{align*}
We now prove a useful lemma which allows to give a well-posed definition of continuity for random fields.
\begin{lemma}\label{lma:measurable sets}
For each $f\in \cl^\infty(\Gcal)$ the sets $RD_f$, $LD_f$, $D_f$, defined above, are $\Gcal$-measurable.
\end{lemma}

\begin{proof}
Observe that the set $RD_f$ can be written as:
\begin{align*}
    RD_f 
    & = \bigcap_{n\in\mathbb{N}}\bigcup_{m\in\mathbb{N}}\left\{\omega \in \Omega : \left(\phi\left(f(\omega)+\frac{1}{n},\omega\right)-\phi(f(\omega),\omega) \right) > \frac{1}{m}\right\} \\ 
    & = \bigcap_{n\in\mathbb{N}}\bigcup_{m\in\mathbb{N}} \left[ \phi\left(f(\cdot)+\frac{1}{n},\cdot\right)-\phi\left(f(\cdot),\cdot \right)\right ]^{-1} \left( \frac{1}{m},+\infty\right) 
\end{align*}
which is $\Gcal$-measurable by measurability of the function
\begin{equation*}
    \omega \to \phi_n(\omega) = \phi\left(f(\omega)+\frac{1}{n},\omega\right)-\phi(f(\omega),\omega).
\end{equation*}
Clearly a similar argument shows that $LD_f\in \Gcal$. Finally, $D_f = LD_f \cup RD_f \in \Gcal$.
\end{proof}

\begin{definition}\label{*-continuity}
The random fields $\phi$ is $\star$-continuous if $\Prob(D_f)=0$ for every $f\in \cl^\infty(\Gcal)$. 
\end{definition}

\begin{remark}\label{remark:star}
Observe that the set $D_f$ defined in Lemma \ref{lma:measurable sets} can be interpreted as: 
\begin{equation*}
    D_f = \left\{\omega \in \Omega : f(\omega) \text{ is a point of discontinuity of the function } \phi(\cdot,\omega)\right\}
\end{equation*}
In particular for any sequence $\{f_n\}_{n\in\mathbb{N}}\subset \cl^\infty(\Gcal)$ such that $f_{n}(\omega)\to f(\omega)$ we have $\phi(f_n(\omega),\omega)\to \phi(f(\omega),\omega)$ for any $\omega\in D_f$.
Moreover it follows that the definition of 
$\star$-continuity is well posed as the set $D_f$ is measurable by Lemma \ref{lma:measurable sets}. 

\noindent Notice also that taking $f \equiv x \in \R$ then 
$D_x=\{\omega \in \Omega : \phi(\cdot,\omega) \text{ is discontinuous in } x\}$. 
Therefore, the condition $\Prob(D_x) = 0$ means that 
for $\Prob$-a.e. $\omega \in \Omega$ the map $\phi(\cdot,\omega)$ is continuous in $x$ . 
On the other hand, if $\phi$ is $\Prob-a.s.$ continuous and satisfies the measurability 
condition of Lemma \ref{lma:measurable sets} then it is also $\star$-continuous. 
Hence, the $\star$-continuity is a notion of continuity which is deeply related to the probability 
space (in particular, to the $\sigma$-algebra) and is weaker than 
the $\Prob$-a.s. continuity of the trajectories but stronger than the $\Prob$-a.s. continuity at fixed points.
\end{remark}

%


\section{State dependent utilities}

As in the rest of the paper $(\Omega,\Fcal)$ denotes a measurable
space and $\cl^{\infty}(\mathcal{F})$ is the space of all
acts, represented by real valued $\Fcal$-measurable and bounded random variables. We here use the term \textquotedblleft act \textquotedblright in order to match the terminology adopted in \cite{WZ99} on which this Appendix is based. This term must be used with care in order to avoid confusion with  the general notion of Anscombe-Aumann acts. Indeed in \cite{AA63} acts are functions from the state space $(\Omega,\Fcal)$ to a convex set of lotteries over a consequence set. 

In this appendix the preference relation is a binary relation
$\succeq$ on $\Fact$ : for $f$, $g$ $\in \Fact$, if $f$ is
preferred to $g$, write $f \succeq g$. The preference relation
satisfies the following axiom:
\begin{description}
\item[(A1)] Preference order: if it is reflexive ($\forall f \in
\Fact$, $f \sim f$), complete ($\forall f,g \in \Fact$, $f \succeq
g$ or $f \preceq g$) and transitive ($\forall f,g,h \in \Fact$
such that $f \succeq g$ and $g \succeq h$ then $f \succeq h$)
\end{description}

\begin{definition}
A representing function of the preference relation is a function
$V : \Fact \to \mathbb{R}$ which is order-preserving, i.e.,
\[ f\succeq g \iff V(f) \ge V(g).
\]
\end{definition}

We use the standard conventions: $f\preceq g$ if $g\succeq f$;
$f\sim f$ if both $g\succeq f$ and $f\succeq g$; $g\nsim f$ if
either $g\nsucceq f$ or $f\nsucceq g$; $g\succ f$ if $g\succeq f$
but $f \nsucceq g$.


\begin{definition}
An event $A\in \Fcal$ is null if $f\mathbf{1}_A +
g\mathbf{1}_{A^c} \sim g$ $\forall f,g \in \Fact$.We shall denote
by $\mathcal{N}(\mathcal{F})$ be the set of null events.
\\ As a consequence a $\succeq$-atom is an element
$A\in\Fcal$ such that for every $B\in \Fcal$ with $\varnothing\neq
B\subset A$ either $B$ or $A\setminus B$ is null.
\\ An event is essential if it belongs to $\Fcal\setminus
\mathcal{N}(\mathcal{F})$.
\end{definition}

We can consider the following additional Axioms:

\begin{description}

\item[(A2)] Strictly monotone if $x\mathbf{1}_A +
f\mathbf{1}_{A^c} \succ y\mathbf{1}_A + f\mathbf{1}_{A^c}$, for
all nonnull events $A\in \Fcal$, for all $f \in \Fact$ and
outcomes $x > y$.

\item [(A3)] Sure-thing principle: consider arbitrary $f, g, h \in
\Fact$ and $A\in \Fcal$ such that $f\mathbf{1}_A +
h\mathbf{1}_{A^c} \preceq g\mathbf{1}_A + h\mathbf{1}_{A^c}$ then
for every $c \in \Fact$ we have $f\mathbf{1}_A + c\mathbf{1}_{A^c}
\preceq g\mathbf{1}_A + c\mathbf{1}_{A^c}$.
\\ (A3) holds on $\Scal(\Fcal)$ if we substitute in the previous
statement $\Fact$ with $\Scal(\Fcal)$ (as defined in the paragraph Notations).

\item[(A4')] Norm continuity if $\forall f \in \Fact$ the sets $\{
g \in \Fact: g \succeq f \}$ and $\{ g \in \Fact : f \succeq g \}$
are $\|\cdot\|_{\infty}$-closed.

\end{description}

\begin{theorem}[Debreu 1960, state-dependent expected utility for finite state space]
\label{Debreu60} Let $\Fact$ the set of acts and $\succeq$ a
preference relation on it. Let the state space $\Omega = \{
\omega_1, ... , \omega_n \}$, where at least three states are
nonnull. Then the following two statements are equivalent:
\begin{enumerate}
\item There exist $n$ continuous functions $V_j : \mathbb{R} \to
\mathbb{R} $, $j = 1, ... , n$, that are strictly increasing for
all nonnull states and constant for all null states, and such that
$\succeq$ is represented by
\begin{equation}
V(f) = \sum_{j=1}^{n} V_j(f(\omega_j)).
\end{equation}
\item $\succeq$ is a norm continuous, strictly monotonic
preference order that satisfies the sure thing principle.
\end{enumerate}
The following uniqueness holds for (1) : $W(f) = \sum_{j=1}^{n}
W_j(f(\omega_j))$ represent $\succeq$ if and only if there exist
$\tau_1, ... , \tau_n \in \mathbb{R}$ and $\sigma > 0 $ such that
$W_j = \tau_j + \sigma V_j$ $\forall j$, implying that $W = \tau +
\sigma V$ for $\tau = \tau_1 + ... + \tau_n$.
\end{theorem}

In \cite{WZ99} the previous Theorem is generalized to an infinite
state spaces $\Omega$ when $\Omega$ contains no atoms. We here
recall the integral reformulation of the Debreu representation
given in \cite{CL06} under pointwise continuity.

\begin{definition} A preference order is
\begin{description}
\item[(A4)] Pointwise continuous if for any uniformly bounded
sequence $\{f_n\}\subseteq \Fact$, such that
$f_n(\omega)\rightarrow f(\omega)$ for any $\omega\in\Omega$ then
$\forall g \in \Fact$ such that $g \succ f$ (resp. $g \prec f$)
$\exists J \in \mathbb{N}$ such that $g \succ f^j$ (resp. $g \prec
f^j$) $\forall j > J$.
\end{description}
\end{definition}

\begin{theorem}[\cite{WZ99}, Theorem 12 and \cite{CL06}, Theorem 5]
\label{formaintegrale} Let $\Fact$ be the set of acts and
$\succeq$ the preference relation on it. Assume that $\Fcal$
contains at least three disjoint essential events. Then the
following two statements are equivalent:
\begin{enumerate}
\item There exists a countably additive measure $\mathbb{P}$ on
$\Omega$ and a function (the state-dependent utility) $u(\omega,
\cdot) : \mathbb{R} \to \mathbb{R}$ strictly increasing $\forall
\omega \in \Omega$, such that $\succeq$ is represented by the
pointwise continuous integral
\[
f \to \int_{\Omega} u(\omega, f(\omega)) d\mathbb{P}.
\]
\item $\succeq$ satisfies: (A1), (A2), (A3) on $\Scal(\Fcal)$,
(A4).
\end{enumerate}
The following uniqueness holds: the couple $(\mathbb{P},u)$ can be
replaced by $(\mathbb{P}^*, u^*)$ if and only if $\mathbb{P}$ and
$\mathbb{P}^*$ are equivalent and $\Prob(u^* = \tau + \sigma
\delta u)=1$, where $\tau: \Omega \to \mathbb{R}$ is
$\Fcal$-measurable, $\sigma > 0$ and $\delta$ is the Radon-Nikodym
density function of $\mathbb{P}$ with respect to $\mathbb{P}^*$.
\end{theorem}


\begin{thebibliography}{99}

\bibitem{AZZ18} Angoshtari, B., Zariphopoulou, T. and Zhou X.Y. (2018) Predictable Forward Performance Processes: The Binomial Case, \textit{preprint}.

\bibitem{AA63}
Anscombe, F.J. and Aumann, R.J. (1963),  A definition of subjective probability,
\textit{The Annals of Mathematical Statistics}, 34(1), 199-205.

\bibitem{Be54}
Bernoulli, D. (1954), Exposition of a new theory on the measurement of risk, trans. L. Sommer, \textit{Econometrica}, 22 , 23-36. Translated from an article originally published in 1738.

\bibitem{CL06} Castagnoli, E. and LiCalzi, M. (2006), Benchmarking real-valued
acts, \textit{Games and Economic Behavior}, 57, 236-253.

\bibitem{De60} Debreu, G. (1960), Topological methods in cardinal utility theory, \textit{Mathematical Methods in Social Sciences}, Stanford University Press, 16-26.

\bibitem{DF31} de Finetti, B. (1931), Sul significato soggettivo della probabilit\`{a}, \textit{Fundamenta Mathematicae}, 298-329.

\bibitem{DJ14} Drapeau S. and Jamneshan, A. (2016), Conditional Preference Orders and their Numerical Representations,
\textit{Jour. Math. Econ.}, 63, 106-118.

\bibitem{DE92} Duffie D. and Epstein, L. (1992), Stochastic Differential Utility, \textit{Econometrica}, 60(2), 353-394. 

\bibitem{EL93} Epstein, L. and LeBreton, M. (1993), Dynamically consistent
beliefs must be Bayesian, \textit{Journal of Economic Theory},
61(1), 1-22.

\bibitem{ES03} Epstein, L. and Schneider, M. (2003),Recursive multiple-priors, \textit{Journal of Economic Theory},
113(1), 1-31.


\bibitem{EW94} Epstein, L. and Wang T. (1994), Interteporal Asset Pricing under Knightian Uncertainty, \textit{Econometrica}, 62(2), 283-322.  

\bibitem{EZ89} Epstein, L. and Zin S. (1989), Substitution, Risk Aversion, and the Temporal Behavior of Consumption and Asset Returns: A Theoretical Framework, \textit{Econometrica}, 54(4), 937-969.  


\bibitem{FS04} F$\ddot{o}$llmer, H. and Shied, A. (2004), \textit{Stochastic Finance. An introduction in discrete
time}, 2nd ed., de Gruyter Studies in Mathematics, Vol.
\textbf{27}.


\bibitem{FM11} Frittelli, M. and Maggis, M. (2011),
Conditional Certainty Equivalent, \textit{Int. J. Theor. Appl.
Fin.}, 14(1), 41-59.

\bibitem{Ka83} Karni, E. (1983), Risk Aversion for State-Dependent Utility Functions: Measurement and Applications, \textit{International Economic Review}, 24(3), 637-647.


\bibitem{KP78} Kreps D.M. and Porteus E.L. (1978), Temporal Resolution of Uncertainty and Dynamic Choice Theory, \textit{Econometrica}, 46(1), 185-200.  


\bibitem{MMR06} Maccheroni, F., M. Marinacci and Rustichini, A.
(2006), Dynamic variational preferences, \textit{Journal of
Economic Theory}, 128(1), 4-44.

\bibitem{Ma15} Marinacci, M. (2015), Model Uncertainty, \textit{Journal of the European Economic Association}, 13, 998-1076.


\bibitem{Me71} Merton, R. (1971), Optimum consumption and portfolio rules in a continuous-time model, \textit{Journal of Economic Theory}, 3(4), 373-413.


\bibitem{MZ06} Musiela, M. and Zariphopoulou, T. (2006),
Investment and valuation under backward and forward dynamic
exponential utilities in a stochastic factor model,
\textit{Advances in Mathematical Finance}, Applied and Numerical
Harmonic Series, 303-334 .


\bibitem{MZ09} Musiela, M. and Zariphopoulou, T. (2009),
Portfolio choice under space-time monotone performance criteria,
\textit{SIAM J. Fin. Math.}, 326-365.

\bibitem{Pe16} Pearl J. (2016), The Sure-Thing Principle, \textit{J. Casual Infer.}, 4(1), 81-86.

\bibitem{Ph09} Pham, H. (2009), \textit{Continuous-time Stochastic Control and Optimization with Financial Applications}, Springer Stochastic Modelling and Applied Probability, Vol.
\textbf{61}.

\bibitem{Pratt} Pratt, A. (1964), Risk
aversion in the Small and in the Large,
\textit{Econometrica}, 32(1/2), 122-136 .

\bibitem{RTV18} Riedel, F., Tallon, J.M. and Vergopoulos V.
(2018), Dynamically Consistent Preferences Under Imprecise
Probabilistic Information, \textit{J. Math. Econ.}, 79, 117-124.

\bibitem{Rock} Rockafellar, R.T. (1968),
Integrals which are convex functionals, \textit{Pacific Journal of
Mathematics} 24(3), 525-539.

\bibitem{Si11} Siniscalchi, M. (2011), Dynamic choice under
ambiguity, \textit{Theoretical Economics}, 6, 379-421.

\bibitem{Sk97} Skiadas, C. (1997), Conditioning and Aggregation of Preferences, \textit{Econometrica}, 65(2), 347-367.

\bibitem{vNM47} von Neumann, J. and Morgenstern, O. (1947), \textit{Theory of Games and Economic Behavior}, 2nd Ed., Princeton University Press.

\bibitem{WZ99} Wakker P.P. and Zank H. (1999), State dependent expected utility for Savage's state space, \textit{Math. Op. Res.}, 24(1), 8-34.


\bibitem{Wa03} Wang T. (2003), Conditional preferences and updating, \textit{Journal of Economic Theory}, 108, 286-321.


\bibitem{ZU16} Zauberman G. and Urminsky O. (2016), Consumer intertemporal
preferences, \textit{Current Opinion in Psychology}, 10, 136-141.
\end{thebibliography}
\end{document}